\documentclass[10pt]{amsart}
      \usepackage[mathscr]{eucal}
      \usepackage{amsmath,amsfonts}
      \parskip\smallskipamount
          \newtheorem{theorem}{Theorem}[section]

      \newtheorem{corollary}[theorem]{Corollary}
      \newtheorem{lemma}[theorem]{Lemma}

      \newtheorem{remark}[theorem]{Remark}
      \makeatletter
      \@addtoreset{equation}{section}
      \makeatother

      \newcommand{\BB}{{\mathbb B}}
      \newcommand{\CC}{{\mathbb C}}

      \newcommand{\DD}{{\mathbb D}}
      \newcommand{\RR}{{\mathbb R}}
      \newcommand{\FF}{{\mathbb F}}
      \newcommand{\TT}{{\mathbb T}}

      \newcommand{\cA}{{\mathcal A}}
      \newcommand{\cB}{{\mathcal B}}
      \newcommand{\cC}{{\mathcal C}}
      \newcommand{\cD}{{\mathcal D}}
      \newcommand{\cE}{{\mathcal E}}
      \newcommand{\cF}{{\mathcal F}}
      \newcommand{\cG}{{\mathcal G}}
      \newcommand{\cH}{{\mathcal H}}
      \newcommand{\cK}{{\mathcal K}}
      \newcommand{\cL}{{\mathcal L}}
      \newcommand{\cM}{{\mathcal M}}
      \newcommand{\cN}{{\mathcal N}}
      \newcommand{\cP}{{\mathcal P}}
      \newcommand{\cR}{{\mathcal R}}
      \newcommand{\cS}{{\mathcal S}}

      \newcommand{\cY}{{\mathcal Y}}
      \newcommand{\cX}{{\mathcal X}}
      
      \newcommand{\cZ}{{\mathcal Z}}

      \newdimen\expt
      \def\boxit#1{\setbox0\hbox{$\displaystyle{#1}$}
            \hbox{\lower.4\expt
       \hbox{\lower3\expt\hbox{\lower\dp0
            \hbox{\vbox{\hrule height.4\expt
       \hbox{\vrule width.4\expt\hskip3\expt
            \vbox{\vskip3\expt\box0\vskip2\expt}%
       \hskip3\expt\vrule width.4\expt}\hrule height.4\expt}}}}}}
      
\begin{document}
       \pagestyle{myheadings}
      \markboth{ Gelu Popescu}{  Free    pluriharmonic
      majorants
        and    noncommutative interpolation  }
      %\pagestyle{plain}
      %\begin{flushright}
       % \it Date of this draft: \today
      %\end{flushright}
      %\bigskip

      \title [     Free    pluriharmonic
      majorants
        and    noncommutative interpolation ]
      {        Free     pluriharmonic
      majorants
        and    noncommutative interpolation
      }
        \author{Gelu Popescu}
     % \date{\today}
\date{November 19, 2007}
      \thanks{Research supported in part by an NSF grant}

      \address{Department of Mathematics, The University of Texas
      at San Antonio \\ San Antonio, TX 78249, USA}
      \email{\tt gelu.popescu@utsa.edu}

\begin{abstract}
 In this paper we initiate the study of sub-pluriharmonic curves
 and free  pluriharmonic majorants on  the noncommutative open  ball
$$
[B(\cH)^n]_1:=\{(X_1,\ldots, X_n)\in B(\cH)^n: \ \|X_1X_1^*+\cdots +
X_nX_n^*\|^{1/2}<1\},
$$
   where $B(\cH)$ is the algebra of
 all bounded linear operators on a Hilbert space $\cH$.
 Several classical results from complex analysis have analogues in
 this noncommutative  multivariable setting.

 We present
 basic properties  for sub-pluriharmonic curves,  characterize the class of sub-pluriharmonic
 curves that admit free pluriharmonic majorants and  find, in
 this case,
 the least free  pluriharmonic majorants. We show that, for any
 free holomorphic function   $\Theta$ on $[B(\cH)^n]_1$, the map
 $$
 \varphi:[0,1)\to C^*(R_1,\ldots, R_n),\quad
 \varphi(r):=\Theta(rR_1,\ldots, rR_n)^*\Theta(rR_1,\ldots, rR_n),
 $$
is a sub-pluriharmonic curve in the Cuntz-Toeplitz algebra
  generated by the right creation operators $R_1,\ldots, R_n$  on the full Fock
space with $n$ generators. We prove that $\Theta$ is in the
noncommutative  Hardy space $H^2_{\bf ball}$  if and only if
$\varphi$  has a free pluriharmonic majorant. In this case, we find
Herglotz-Riesz  and Poisson  type representations for the least
pluriharmonic majorant of $\varphi$. Moreover, we obtain a
characterization of  the unit ball of $H^2_{\bf ball}$ and provide a
parametrization and concrete representations for all free
pluriharmonic majorants of $\varphi$, when $\Theta$ is in the unit
ball of $H^2_{\bf ball}$.

In the second part of this paper, we introduce   a generalized
noncommutative commutant lifting ({\bf GNCL})  problem which
extends, to our noncommutative multivariable setting, several
lifting problems including the classical Sz.-Nagy--Foia\c s
commutant lifting problem and the extensions obtained by
Treil-Volberg, Foia\c s-Frazho-Kaashoek, and Biswas-Foia\c s-Frazho,
as well as their multivariable noncommutative versions. We solve the
{\bf GNCL} problem and, using the results regarding
sub-pluriharmonic curves
 and free pluriharmonic majorants on noncommutative balls,
 we provide a complete   description of all solutions.
 In particular,
 we obtain   a concrete  Schur type  description of all solutions in  the noncommutative commutant lifting
 theorem.
\end{abstract}

      \maketitle

\bigskip

\bigskip

\section*{Introduction}

Noncommutative multivariable operator theory has received a lot of
attention,  in the last two decades, in the attempt of   obtaining a
{\it free}
  analogue of
  Sz.-Nagy--Foia\c s theory \cite{SzF-book}, for row contractions,
  i.e., $n$-tuples of bounded linear operators $(T_1,\ldots, T_n)$ on a Hilbert
  space such that
  $$T_1T_1^*+\cdots +T_nT_n^*\leq I.
  $$
  Significant progress  has been
made regarding
 noncommutative dilation theory and its applications to
    interpolation in several variables (\cite{F}, \cite{Bu}, \cite{Po-models},
  \cite{Po-isometric}, \cite{Po-charact}, \cite{Po-intert}, \cite{DKS},
   \cite{Po-unitary},  \cite{Po-varieties},
    \cite{Po-analytic},
\cite{Po-interpo}, \cite{ArPo2}, \cite{AMc3}, \cite{Po-structure},
\cite{DP}, \cite{BTV}, \cite{Po-nehari}, \cite{Po-entropy},
\cite{Po-commutator},  \cite{Po-pluriharmonic},
\cite{Po-free-hol-interp}, etc.), and unitary invariants for
$n$-tuples of operators (\cite{Po-charact}, \cite{Arv}, \cite{Arv2},
\cite{Po-curvature},
   \cite{Po-similarity}, \cite{BT},
 \cite{Po-entropy}, \cite{Po-unitary},  etc.).

In \cite{Po-holomorphic} and  \cite{Po-pluriharmonic}, we developed
a theory of holomorphic (resp. pluriharmonic)  functions in several
noncommuting (free) variables and provide a framework for  the study
of arbitrary
 $n$-tuples of operators on a Hilbert space. Several classical
 results from complex analysis have
 free analogues in
 this  noncommutative multivariable setting. This theory enhances
 our program to develop a {\it free}
  analogue of
  Sz.-Nagy--Foia\c s theory. In related areas of research,
  we remark
 the work  of Helton,  McCullough, Putinar, and  Vinnikov, on symmetric noncommutative  polynomials (\cite{He}, \cite{He-C},
 \cite{He-C-P},
 \cite{He-C-P2}, \cite{He-C-V}), and the work of Muhly and Solel  on
 representations of tensor algebras over
 $C^*$ correspondences (see \cite{MuSo1}, \cite{MuSo2}).

 The present paper is a natural  continuation of \cite{Po-holomorphic} and
 \cite{Po-pluriharmonic}.
 We initiate here the study of sub-pluriharmonic  curves
 and free  pluriharmonic majorants on noncommutative balls.
 We are lead to a characterization of the noncommutative Hardy space
 $H^2_{\bf ball}$ in terms of  free pluriharmonic majorants, and to a
 Schur type description of the unit ball of $H^2_{\bf ball}$.
 These results are used to solve a multivariable  commutant lifting
 problem and provide a description of all solutions.

To put our work in perspective and present our results, we need to
set up some notation. Let $\FF_n^+$ be the unital free semigroup on
$n$ generators $g_1,\ldots, g_n$ and the identity $g_0$.  The length
of $\alpha\in \FF_n^+$ is defined by $|\alpha|:=0$ if $\alpha=g_0$
and $|\alpha|:=k$ if
 $\alpha=g_{i_1}\cdots g_{i_k}$, where $i_1,\ldots, i_k\in \{1,\ldots, n\}$.
If $(X_1,\ldots, X_n)\in B(\cH)^n$, where $B(\cH)$ is the algebra of
all bounded linear operators on  an infinite dimensional Hilbert
space $\cH$, we denote $X_\alpha:= X_{i_1}\cdots X_{i_k}$ and
$X_{g_0}:=I_\cH$.
We recall  (see \cite{Po-holomorphic}, \cite{Po-pluriharmonic}) that
a map $f:[B(\cH)^n]_1\to B(\cH)$ is called free holomorphic function
with scalar coefficients if it
 has a representation
$$
f(X_1,\ldots, X_n)=\sum_{k=0}^\infty \sum_{|\alpha|=k}
 a_\alpha X_\alpha,\qquad X:=(X_1,\ldots, X_n)\in [B(\cH)^n]_1,
$$
where $\{a_\alpha\}_{\alpha\in \FF_n^+}$ are complex numbers  with \
$\limsup\limits_{k\to\infty} (\sum\limits_{|\alpha|=k}
|a_\alpha|^2)^{1/2k}\leq 1$.  The function $f$ is in the
noncommutative Hardy space $H_{\bf ball}^\infty$  if
$$\|f\|_\infty:=\sup\limits_{X\in [B(\cH)^n]_1} \|f(X)\|<\infty,
$$
 and   in $H_{\bf ball}^2$ if $\|f\|_2:=(\sum_{\alpha\in
\FF_n^+} |a_{\alpha}|^2)^{1/2}< \infty$.

We say that $h:[B(\cH)^n]_1\to B(\cH)$ is a self-adjoint free
pluriharmonic function on $[B(\cH)^n]_1$ if $h=\text{\rm Re}\, f$
for some free holomorphic function $f$. An arbitrary free
pluriharmonic function is a linear combination of self-adjoint free
pluriharmonic functions.  In the particular case when $n=1$, a
function $X\mapsto h(X)$ is free pluriharmonic  on $[B(\cH)]_1$ if
and only if the function $\lambda\mapsto h(\lambda)$ is harmonic on
the open unit disc $\DD:=\{\lambda\in\CC:\ |\lambda|<1\}$. If $n\geq
2$ and $h$ is a free pluriharmonic function on $[B(\cH)^n]_1$, then
its scalar representation $(z_1,\ldots, z_n)\mapsto h(z_1,\ldots,
z_n)$ is a pluriharmonic function  on the open unit ball
$\BB_n:=\{z\in \CC^n:\ \|z\|_2<1\}$, but the converse is not true.
As shown in \cite{Po-holomorphic}, \cite{Po-pluriharmonic}, several
classical
 results from complex analysis, regarding holomorphic (resp.~harmonic) functions, have
 free analogues in
 this  noncommutative multivariable setting.

Let $H_n$ be an $n$-dimensional complex  Hilbert space with
orthonormal
      basis
      $e_1$, $e_2$, $\dots,e_n$, where $n\in\{1,2,\dots\}$.
       We consider the full Fock space  of $H_n$ defined by
      $$F^2(H_n):=\CC 1\oplus \bigoplus_{k\geq 1} H_n^{\otimes k},$$
      where   $H_n^{\otimes k}$ is the (Hilbert)
      tensor product of $k$ copies of $H_n$.
      Define the left  (resp.~right) creation
      operators  $S_i$ (resp.~$R_i$), $i=1,\ldots,n$, acting on $F^2(H_n)$  by
      setting
      $$
       S_i\varphi:=e_i\otimes\varphi, \quad  \varphi\in F^2(H_n),
      $$
       (resp.~$
       R_i\varphi:=\varphi\otimes e_i, \quad  \varphi\in F^2(H_n)
      $).
The noncommutative disc algebra $\cA_n$ (resp.~$\cR_n$) is the norm
closed algebra generated by the left (resp.~right) creation
operators and the identity. The   noncommutative analytic Toeplitz
algebra $F_n^\infty$ (resp.~$\cR_n^\infty$)
 is the the weakly
closed version of $\cA_n$ (resp.~$\cR_n$). These algebras were
introduced in \cite{Po-von} in connection with a noncommutative
version of the classical  von Neumann inequality \cite{von} (see
\cite{Pi-book} for a nice survey). They
 have  been studied
    in several papers
\cite{Po-charact},  \cite{Po-multi},  \cite{Po-funct},
\cite{Po-analytic}, \cite{Po-disc}, \cite{Po-poisson},
 \cite{ArPo2}, and recently in
  \cite{DP1}, \cite{DP2},   \cite{DP},
   \cite{Po-curvature},  \cite{DKP},
   \cite{Po-similarity}, \cite{DLP},   and  \cite{Po-holomorphic}.

In Section 1,  we introduce  the    class of sub-pluriharmonic
curves
  and present
 basic properties. We prove that a self-adjoint (i.e.  $g(r)=g(r)^*$) map $g:[0,1)\to
\overline{ \cA_n^*+ \cA_n}^{\|\cdot\|}$   is a sub-pluriharmonic
curve in the Cuntz-Toeplitz  $C^*$-algebra $C^*(S_1,\ldots, S_n)$
(see \cite{Cu}) if and only if
$$
g(r)\leq { P}_{\frac{r}{\gamma}S}[g(\gamma)]\qquad \text{ for  }\
0\leq r<\gamma<1,
$$
where $P_X[u]$ is the noncommutative Poisson transform of $u$ at
$X$. We obtain a characterization  for  the class of all
sub-pluriharmonic curves
  that admit free
pluriharmonic majorants, and   prove the existence of the least
pluriharmonic majorant. More precisely, we show that  a self-adjoint
mapping  $g:[0,1)\to \overline{ \cA_n^*+ \cA_n}^{\|\cdot\|}$ has a
pluriharmonic majorant  if and only if
\begin{equation*}
\sup_{0<r<1} \|  \tau [g(r)]\|<\infty,
\end{equation*}
where $\tau$ is the linear functional on $B(F^2(H_n))$ defined by
$\tau(f):=\left< f(1),1\right>$.
 In this case, there is a least
pluriharmonic majorant for $g$, namely, the map $$ [0,1)\ni r\to
u(rS_1,\ldots, rS_n)\in \overline{ \cA_n^*+ \cA_n}^{\|\cdot\|},
$$
where the free pluriharmonic function $u$ is given by
$$
u(X_1,\ldots, X_n):=\lim_{\gamma\to 1}{
P}_{\frac{1}{\gamma}X}[g(\gamma)]
$$
for any $X:=(X_1,\ldots, X_n)\in [B(\cH)^n]_1$ and  the limit is in
the norm topology.

Let $\theta$ be an analytic function on the open disc $\DD$. It is
well-known that  the map $\varphi:\DD\to \RR^+$ defined by
$\varphi(\lambda):=|\theta(\lambda)|^2$ is subharmonic. A  classical
result on harmonic majorants (see Section 2.6 in \cite{Du}) states
that $\theta$ is in the Hardy space $H^2(\DD)$ if and only if
$\varphi$ has a harmonic majorant. Moreover, the least harmonic
majorant of $\varphi$ is given by the formula
$$
h(\lambda)=\frac{1}{2\pi}\int_0^{2\pi}
\frac{e^{it}+\lambda}{e^{it}-\lambda}|\theta(e^{it})|^2 dt,\quad
\lambda\in \DD.
$$
In Section 2, we obtain free analogues of these results.  We show
that, for any
 free holomorphic function   $\Theta$ on the noncommutative ball $[B(\cH)^n]_1$, the mapping
 $$
 \varphi:[0,1)\to C^*(R_1,\ldots, R_n),\quad
 \varphi(r)=\Theta(rR_1,\ldots, rR_n)^*\Theta(rR_1,\ldots, rR_n),
 $$
is a sub-pluriharmonic curve in the Cuntz-Toeplitz algebra
  generated by the right creation operators $R_1,\ldots, R_n$  on the full Fock
space with $n$ generators. We prove that a free holomorphic function
$\Theta$   is in the noncommutative Hardy space $H^2_{\bf ball}$ if
and only if $\varphi$  has a pluriharmonic majorant. In this case,
the least pluriharmonic majorant $\psi$  for $\varphi$ is given by $
\psi(r):=\text{\rm Re}\, W(rR_1, \ldots rR_n)$, $r\in[0,1), $
  where $W$ is the free holomorphic function  having the
  Herglotz-Riesz (\cite{Her}, \cite{Ri})
  type representation
  \begin{equation*}
  W(X_1,\ldots,
  X_n)=({\mu}_\theta\otimes \text{\rm id})\left[\left(I+\sum_{i=1}^n R_i^*\otimes
  X_i\right)\left(I-\sum_{i=1}^n R_i^*\otimes
  X_i\right)^{-1}\right]
  \end{equation*}
for $(X_1,\ldots, X_n)\in [B(\cH)^n]_1$, where
$\mu_\theta:\cR_n^*+\cR_n\to  \CC$ is a   positive linear map
uniquely determined by  $\Theta$. Poisson type representations for
the least pluriharmonic majorant are also considered.

In  Section 3,  we   provide a parametrization and concrete
representations for all pluriharmonic majorants of   the
sub-pluriharmonic curve $ \varphi$,
 where  $\Theta$ is in the unit ball of  $H^2_{\bf ball}$.
 We show that, up to a normalization,  all pluriharmonic
 majorants of  $ \varphi$  have the form $\text{\rm Re}\, F$,
 where $F$ is a free holomorphic function  given by
 \begin{equation*}
  F(X)=W(X) + \left(D_\Gamma \otimes
I\right)[I+G(X)][I-G(X)]^{-1}\left(D_\Gamma \otimes I\right), \quad
X \in [B(\cH)^n]_1,
\end{equation*}
  where $G$ is in the unit ball  of the noncommutative
Hardy space  $ H_{\bf ball}^\infty$ with coefficients in $
B(\cD_\Gamma)$, $\Gamma$ is the symbol of $\Theta$, and $\text{\rm
Re}\, W$  is  the least pluriharmonic majorant of $\varphi$
considered above.
 Moreover, $F$ and $G$ uniquely determine
each other. Using this result, we obtain a characterization of  the
unit ball of $H^2_{\bf ball}$. More precisely, we show that
 $\Theta$ is  a free holomorphic function
in the unit ball of  $H_{\text{\bf ball}}^2  $
  if and only if it has a  Schur type representation
\begin{equation*}
  \Theta(X)=L(X)\left[ I_{  \cH}-
\sum_{i=1}^n  X_iM_i(X)\right]^{-1}\quad  \text{ for
  } \ X:=(X_1,\ldots, X_n)\in [B(\cH)^n]_1,
\end{equation*}
 where
$\left[\begin{matrix} L & M_1& \cdots & M_n\end{matrix}\right]^t$
 (${}^t$ is the transpose) is a free holomorphic function in the unit ball of the noncommutative Hardy
space $H^\infty_{\bf ball}$ with
  coefficients in $\CC^n$.

We should mention that all the results  of this paper are obtained
in the   more general setting  of sub-pluriharmonic curves
(resp.~free pluriharmonic majorants) with operator-valued
coefficients.

The celebrated  {\it commutant lifting theorem} ({\bf CLT}),  due to
Sz.-Nagy and Foia\c s  \cite{SzF} for the general case and Sarason
\cite{S} for an important special case, provides the natural
geometric framework for classical  and modern $H^\infty$
interpolation problems. For a detailed analysis of the {\bf CLT} and
its applications to various interpolation and extension problems, we
refer the reader to  the monographs \cite{FF-book} and
\cite{FFGK-book}.

In the present paper, we introduce   a {\it generalized
noncommutative commutant lifting} ({\bf GNCL})  problem, which
extends to a multivariable setting  several lifting problems
including the classical Sz.-Nagy--Foia\c s commutant lifting
\cite{SzF} and the extensions obtained by Treil-Volberg \cite{TV},
Foia\c s-Frazho-Kaashoek \cite{FFK}, and Biswas-Foia\c s-Frazho
\cite{BFF}, as well as  their multivariable noncommutative versions
\cite{Po-isometric}, \cite{Po-intert}, \cite{Po-interpo},
\cite{Po-nehari}, \cite{Po-entropy}, \cite{Po-commutator}. While, in
the classical case, there is a large literature regarding
parametrizations  and Schur (\cite{Sc})  type  representations  of
the set of all solutions in  the {\bf CLT} (see \cite{FF-book},
\cite{FFGK-book}, \cite{FHK}, \cite{FFK}, etc.),  very little is
known in the noncommutative multivariable case (see
\cite{Po-entropy}). In the present paper, we try to fill in  this
gap.

A lifting {\it data set} $\{A,T,V,C,Q\}$  for the   {\bf GNCL}
problem is defined as follows.
  Let $T:=(T_1,\dots, T_n)$, $T_i\in B(\cH)$, be a row
contraction and let $V:=(V_1,\ldots, V_n)$, $V_i\in B(\cK)$, be the
minimal isometric dilation of  $T$ on a Hilbert space $\cK\supset
\cH$, in the sense of \cite{Po-isometric}.
  Let $Q:=(Q_1,\ldots, Q_n)$,  $Q_i\in B(\cG_i,\cX)$,  and
 $C:=(C_1,\ldots, C_n)$, $C_i\in
B(\cG_i,\cX)$,  be such  that
\begin{equation*}
 \left[\delta_{ij} C_i^* C_j\right]_{n\times n}\leq \left[ Q_i^*
 Q_j\right]_{n\times n}.
 \end{equation*}
Consider a contraction  $A\in B(\cX,\cH)$   such that
\begin{equation*}
  T_i AC_i=AQ_i, \qquad i=1,\ldots,n.
\end{equation*}
We say that   $B$ is a {\it contractive interpolant} for $A$ with
respect to $\{A,T,V,C,Q\}$ if  $B\in B(\cX, \cK)$  is a contraction
satisfying the conditions
\begin{equation*}
  P_\cH B=A\quad \text{ and } \
 V_iBC_i=BQ_i,\quad i=1,\ldots,n,
\end{equation*}
where $P_\cH$ is the orthogonal projection from $\cK$ onto $\cH$.
The {\bf GNCL} problem is  to find   contractive interpolants $B$
of $A$ with respect to the data set $\{A,T,V,C,Q\}$.

In Section 4, we  solve the {\bf GNCL}  problem  and, using our
results regarding sub-pluriharmonic curves
 and  free pluriharmonic majorants on noncommutative balls,
 we provide a Schur type description of all solutions in terms of the elements
 of the unit ball of an appropriate noncommutative Hardy space
 $H^\infty_{\bf ball}$ (see Theorem \ref{main}). Our results are new, in
 particular,
 even in the   multivariable settings considered in \cite{Po-isometric}, \cite{Po-intert},
  \cite{Po-interpo},
\cite{Po-nehari}, \cite{Po-entropy}, and \cite{Po-commutator}. Here
we point out  some remarkable special cases.

 If  the data set $\{A,T,V,C,Q\}$ is
 such that $\cG_i=\cX$, $C_i=I_\cX$, $Q_i=Y_i\in
B(\cX)$,  for each $i=1,\ldots,n$,  and $Y:=(Y_1,\ldots, Y_n)$ is a
row isometry,
  we obtain (see Theorem \ref{NCLT}) a parametrization
 and a concrete Schur type  description of all solutions in the {\it noncommutative commutant lifting theorem}
 (\cite{Po-isometric}).  On the other hand,  when $\cG_i=\cX$ and
$C_i=I_\cX$ for $i=1,\ldots, n$,  we obtain  a description of all
solutions of the multivariable version \cite{Po-nehari} of
Treil-Volberg commutant lifting theorem \cite{TV}. The {\bf GNCL}
theorem  also implies  the multivariable version \cite{Po-nehari} of
the weighted commutant lifting theorem of Biswas-Foia\c s-Frazho
\cite{BFF}.

 Finally, we remark that,
when applied to the interpolation theory setting, our results
provide
    complete
 descriptions  of all solutions to the Nevanlinna-Pick (\cite{N}, \cite{Pic}),
 Carath\'eodory-Fej\' er  (\cite{Ca}, \cite{CaFe}), and Sarason (\cite{S}) type interpolation problems  for
  the noncommutative Hardy spaces $H^\infty_{\bf
 ball}$  and   $H^2_{\bf ball}$, as well as consequences to  (norm constrained)
 interpolation on the unit ball of $\CC^n$. These issues will be
 addressed  in a future paper.

\bigskip

\section{ Sub-pluriharmonic curves  in Cuntz-Toeplitz  algebras }

In this section we initiate the study    of sub-pluriharmonic curves
and present some basic properties.  We obtain a characterization for
the class of  all sub-pluriharmonic curves which admit free
pluriharmonic majorants, and  find  the least pluriharmonic
majorants. Some of the results of this section can be extended to
the class of $C^*$-subharmonic curves.

 We need to recall from \cite{Po-pluriharmonic}  a few facts concerning the {\it noncommutative
Berezin transform} associated with a  completely bounded linear  map
$\mu: B(F^2(H_n))\to B(\cE)$, where $\cE$ is a separable Hilbert
space. This is the map
$$\cB_\mu:
B(F^2(H_n))\times [B(\cH)^n]_1\to B(\cE)\otimes_{min} B(\cH)$$
defined by
\begin{equation*}
 \cB_\mu(f,X):=\widetilde{\mu}\left[ B_X^*(f\otimes
I_\cH)B_X\right], \qquad f\in B(F^2(H_n)), \ X:=(X_1,\ldots,X_n)\in
[B(\cH)^n]_1,
\end{equation*}
  where the operator  $B_X \in B(F^2(H_n))\otimes
\cH)$  is given
  by
\begin{equation*}
B_X := (I_{F^2(H_n)} \otimes \Delta_X)
 \left(I-R_1\otimes X_1^*-\cdots -R_n\otimes X_n^*\right)^{-1},
\end{equation*}
$\Delta_X:=(I_\cH-\sum_{i=1}^n X_iX_i^*)^{1/2}$, and
$$\widetilde
\mu:=\mu\otimes \text{\rm id} : B(F^2(H_n)) \otimes_{min} B(\cH)\to
B(\cE)\otimes_{min} B(\cH)
$$
is the completely bounded linear map  uniquely defined by   $
\widetilde \mu(f\otimes Y):= \mu(f)\otimes Y$ for $f\in B(F^2(H_n))
$ and $Y\in B(\cH)$.

An important particular case is the Berezin transform $B_\tau$,
where $\tau$ is the linear functional on $B(F^2(H_n))$ defined by
$\tau(f):=\left<f(1),1\right>$. This  will be called Poisson
transform because it coincides with the {\it noncommutative
 Poisson transform} introduced
in \cite{Po-poisson}. More precisely,  we have
 $$\cB_\tau(f, X)=P_X(f):=K_X^*(f\otimes I)K_X,
 $$
 where  $K_X=B_X |_{1\otimes \cH}:\cH\to F^2(H_n)\otimes
\cH$ is the noncommutative Poisson kernel.
 We recall  from \cite{Po-poisson} that the restriction of $P_X$ to the $C^*$-algebra
$C^*(S_1,\ldots, S_n)$, generated by the left creation operators,
can be extended to the closed
 ball $[B(\cH)^n]_1^-$ by setting
 \begin{equation*}
 P_X[f]:=\lim_{r\to 1} K_{rX}^* (f\otimes I)K_{rX},\qquad X\in
 [B(\cH)^n]_1^-,\ f\in C^*(S_1,\ldots, S_n),
 \end{equation*}
 where $rX:=(rX_1,\ldots, rX_n)$, $r\in (0,1)$,  and
    the limit exists in the operator
 norm topology of $B(\cH)$.
In this case, we have
 \begin{equation*}
P_X(S_\alpha S_\beta^*)=X_\alpha X_\beta^*\ \text{ for any
  } \  \alpha,\beta\in \FF_n^+.
  \end{equation*}
 When $X:=(X_1,\ldots, X_n)$  is a pure $n$-tuple, i.e., $\sum_{|\alpha|=k} X_\alpha
 X_\alpha^*\to 0$,  as $k\to\infty$,  in the strong operator topology,
 then  we have $P_X(f)=K_X^*(f\otimes I)K_X$.
The {\it operator-valued Poisson transform} at $X\in [B(\cH)^n]_1$
is
 the map
${\bf P}_X:B(\cE\otimes F^2(H_n))\to B(\cE\otimes \cH)$  defined by
\begin{equation*}
 {\bf P}_X[u]:=(I_\cE\otimes K_X^*)(u\otimes
I_\cH)(I_\cE\otimes K_X)
\end{equation*}
for any $u\in B(\cE\otimes F^2(H_n))$.
  We refer to \cite{Po-poisson}, \cite{Po-curvature},
\cite{Po-similarity}, and \cite{Po-unitary} for more on
noncommutative Poisson transforms on $C^*$-algebras generated by
isometries.

Let $\cP_n$ be the set of all polynomials  in $S_1,\ldots, S_n$ and
the identity, and let $\text{\boldmath{$\cP$}}_n$ denote (throughout
the paper) the spatial tensor product $B(\cE)\otimes \cP_n$. A  {\it
pluriharmonic   curve} in the (spatial)  tensor product
 $B(\cE)\otimes_{min} C^*(S_1,\ldots, S_n)$ is a  map $\varphi:[0,\gamma)\to
\overline{\text{\boldmath{$\cP$}}_n^*+\text{\boldmath{$\cP$}}_n}^{\|\cdot\|}$
satisfying the Poisson mean value property, i.e.,
\begin{equation*}
  \varphi(r)= {\bf P}_{\frac{r}{t} S}[\varphi(t)]\quad \text{
for } \ 0\leq r<t<\gamma,
\end{equation*}
where $S:=(S_1,\ldots, S_n)$.   According to
\cite{Po-pluriharmonic}, there exists a one-to-one correspondence
$u\mapsto \varphi$ between the set of all free  pluriharmonic
functions on the noncommutative ball of radius $\gamma$,
$[B(\cH)^n]_\gamma$, and the set of all pluriharmonic curves
$\varphi:[0,\gamma)\to
\overline{\text{\boldmath{$\cP$}}_n^*+\text{\boldmath{$\cP$}}_n}^{\|\cdot\|}$
 in the tensor product
 $B(\cE)\otimes_{min} C^*(S_1,\ldots, S_n)$. Moreover, we have
$$u(X)={\bf P}_{\frac{1}{r} X}[\varphi(r)]\quad \text{
for } \   X\in [B(\cH)^n]_r \ \text{ and } \ r\in (0,\gamma),
$$
and $\varphi(r)=u(rS_1,\ldots, rS_n)$ if $r\in [0,\gamma)$.  We also
 solved  in \cite{Po-pluriharmonic}  a Dirichlet type extension problem which
 implies
  that a   free pluriharmonic function  $u$   has continuous extension to the closed ball
$[B(\cH)^n]_\gamma^-$ if and only if
 there exists  a pluriharmonic curve    $\varphi:[0,\gamma]\to
\overline{\text{\boldmath{$\cP$}}_n^*
+\text{\boldmath{$\cP$}}_n}^{\|\cdot\|}$ such that $u(X)={\bf
P}_{\frac{1}{r} X}[\varphi(r)]$ for any $ X\in [B(\cH)^n]_r$  and  $
r\in (0,\gamma]$. We add that $u$ and $\varphi$ uniquely determine
each other and satisfy the equations $u(rS_1,\ldots,
rS_n)=\varphi(r)$ if $r\in [0,\gamma)$ and
$\varphi(\gamma)=\lim\limits_{r\to \gamma} u(rS_1,\ldots, rS_n)$,
where the convergence is in the operator norm topology. Now, we can
introduce  the class of {\it sub-pluriharmonic curves} in the tensor
algebra $B(\cE)\otimes_{min} C^*(S_1,\ldots, S_n)$. We say that a
map
$$\psi:[0,1)\to \overline{\text{\boldmath{$\cP$}}_n+
\text{\boldmath{$\cP$}}_n}^{\|\cdot\|}\quad \text{ with }\ \psi(r)=\psi(r)^*, r\in
[0,1), $$
 is  a {\it  sub-pluriharmonic} curve   provided that  for each $\gamma\in (0,1)$
 and each self-adjoint  pluriharmonic  curve
  $\varphi: [0,\gamma]\to \overline{\text{\boldmath{$\cP$}}_n^* +
  \text{\boldmath{$\cP$}}_n}^{\|\cdot\|}$,
  if
 $\psi(\gamma)\leq \varphi(\gamma)$,   then
 $$
\psi(r)\leq \varphi(r)\ \text{ for any } \ r\in [0,\gamma].
 $$

Our first result is the following characterization  of
sub-pluriharmonic curves.

\begin{theorem}
\label{chara-ine} Let $g:[0,1)\to
\overline{\text{\boldmath{$\cP$}}_n^*+\text{\boldmath{$\cP$}}_n}^{\|\cdot\|}$
be a map with $g(r)=g(r)^*$ for $r\in [0,1)$. Then $g$ is
sub-pluriharmonic    if and only if
$$
g(r)\leq {\bf P}_{\frac{r}{\gamma}S}[g(\gamma)]\qquad \text{ for  }\
0\leq r<\gamma<1.
$$
The equality holds if $g$ is   pluriharmonic.
\end{theorem}
\begin{proof}
Assume that $g$ is sub-pluriharmonic  and let $0\leq r<\gamma<1$.
Since $g(\gamma)\in
\overline{\text{\boldmath{$\cP$}}_n^*+\text{\boldmath{$\cP$}}_n}^{\|\cdot\|}$,
one can use Theorem 4.1 from \cite{Po-pluriharmonic}, to deduce that
the map $u:[B(\cH)^n]_\gamma\to B(\cE)\otimes_{min} B(\cH)$ defined
by
\begin{equation*}
 u(X)={\bf P}_{\frac{1}{\gamma} X}[g(\gamma)]\quad \text{ for } \
X\in [B(\cH)^n]_\gamma
\end{equation*}
is free pluriharmonic on $[B(\cH)^n]_\gamma$ and   has a continuous
extension to $[B(\cH)^n]_\gamma^-$.  In this case, we have
$g(\gamma)=\lim_{t\to \gamma} u(tS_1,\ldots, tS_n)$.  Moreover, if
$\varphi:[0,\gamma]\to
\overline{\text{\boldmath{$\cP$}}_n^*+\text{\boldmath{$\cP$}}_n}^{\|\cdot\|}$
is  the pluriharmonic curve uniquely associated with  the free
pluriharmonic  function $u$, then
$$\varphi(r)={\bf P}_{\frac{r}{\gamma} S}[g(\gamma)]\quad \text{ for } \
 r\in [0,\gamma].
 $$
Since $g$ is a sub-pluriharmonic curve and
$g(\gamma)=\varphi(\gamma)$, we deduce that
$$
g(r)\leq \varphi(r)= {\bf P}_{\frac{r}{\gamma}S}[g(\gamma)]\quad
\text{ for any }\  r\in [0,\gamma].
$$

Conversely, assume that  $g$ has the property that
\begin{equation}
\label{uP}
 g(r)\leq {\bf P}_{\frac{r}{\gamma}S}[g(\gamma)]\quad
\text{ for any }\ 0\leq r<\gamma<1.
\end{equation}
Let $\varphi:[0,\gamma]\to
\overline{\text{\boldmath{$\cP$}}_n^*+\text{\boldmath{$\cP$}}_n}^{\|\cdot\|}$
be  a pluriharmonic curve such that   $\varphi(r)=\varphi(r)^*$ for
$r\in [0,\gamma]$, and assume that  $g(\gamma)\leq \varphi(\gamma)$.
Since $g(\gamma)$ and $\varphi(\gamma)$ are in $B(\cE)\otimes_{min}
C^*(S_1,\ldots, S_n)$ and the noncommutative Poisson transform  is a
positive map, we deduce that
\begin{equation}
\label{P<P} {\bf P}_{\frac{r}{\gamma}S}[g(\gamma)]\leq {\bf
P}_{\frac{r}{\gamma}S}[\varphi(\gamma)],\qquad  0\leq r<\gamma.
\end{equation}
 On the other
hand, since $\varphi$ is a  pluriharmonic curve on $[0,\gamma]$, we
have $\varphi(r)={\bf P}_{\frac{r}{\gamma}S}[\varphi(\gamma)]$ for
$0\leq r<\gamma$. Hence, using relations \eqref{uP} and \eqref{P<P},
we deduce that
 $g(r)\leq \varphi(r)$ for any $r\in [0,\gamma]$.
The proof is complete.
\end{proof}

As a consequence of Theorem \ref{chara-ine},  we remark that if
$u_1,\ldots, u_k$ are sub-pluriharmonic curves and
$\lambda_1,\ldots, \lambda_k$ are positive numbers then $\lambda_1
u_1+\cdots +\lambda_k u_k$ is sub-pluriharmonic.  Notice also that a
self-adjoint function $u:[0,1)\to
\overline{\text{\boldmath{$\cP$}}_n^*+\text{\boldmath{$\cP$}}_n}^{\|\cdot\|}$
is pluriharmonic if and only if both $u$  and $-u$ are
sub-pluriharmonic.

We recall   (see   Lemma 2.3 from \cite{Po-pluriharmonic}) that if
$\gamma_1>0$ and \ $0\leq\gamma_j\leq 1$ for $j=2,\ldots, k$, then
the noncommutative Poisson transform has the property
$$
P_{\gamma_1\cdots \gamma_k X}=P_{\gamma_1 X}\circ P_{\gamma_2
S}\circ \cdots \circ P_{\gamma_k S}
$$
for any $X\in [B(\cH)^n]_{\frac{1}{\gamma_1}}$, where
$S:=(S_1,\ldots, S_n)$ is the $n$-tuple of left creation operators
on the Fock space $F^2(H_n)$. Moreover, we have
\begin{equation}
\label{PPPP}
 {\bf P}_{\gamma_1\cdots \gamma_k X}[g]=\left({\bf
P}_{\gamma_1 X}\circ {\bf P}_{\gamma_2 S}\circ \cdots \circ {\bf
P}_{\gamma_k S}\right)[g]
\end{equation}
for any $g\in B(\cE)\otimes_{min} B(F^2(H_n))$.

\begin{corollary}
\label{ine-subha} Let $g:[0,1)\to
\overline{\text{\boldmath{$\cP$}}_n^*+\text{\boldmath{$\cP$}}_n}^{\|\cdot\|}$
be a sub-pluriharmonic curve  and let  $\tau$ be the linear
functional on $B(F^2(H_n))$ defined by
$\tau(f):=\left<f(1),1\right>$. Then
\begin{enumerate}
\item[(i)] $ {\bf P}_{\frac{r}{\gamma_1}S}[g(\gamma_1)]\leq {\bf
P}_{\frac{r}{\gamma_2}S}[g(\gamma_2)]$ \  for  $ 0< r<
\gamma_1<\gamma_2<1; $
 \item[(ii)]
$g(0)\leq \widetilde \tau [g(\gamma_1)]\leq \widetilde \tau
[g(\gamma_2)]$ \ for $0 < \gamma_1<\gamma_2<1$, where $\widetilde
\tau:=\tau\otimes \text{\rm id}$;
\item[(iii)]  $\widetilde \tau [g(0)]\leq \widetilde \tau
[g(\gamma_1)]$ \  for $0\leq \gamma_1<1$.
\end{enumerate}
\end{corollary}
\begin{proof}
According to Theorem \ref{chara-ine}, we have
\begin{equation}
\label{gP} g(r)\leq {\bf P}_{\frac{r}{\gamma_2}S}[g(\gamma_2)]\quad
\text{ for  }\ 0\leq r<\gamma_2<1,
\end{equation}
which implies
\begin{equation*}
  g(\gamma_1)\leq {\bf
P}_{\frac{\gamma_1}{\gamma_2}S}[g(\gamma_2)]\quad \text{ for  }\
0\leq r<\gamma_1<\gamma_2<1.
\end{equation*}
Hence, using  \eqref{PPPP} and the positivity of the noncommutative
Poisson transform, we deduce that
\begin{equation}
\label{PPPP2}
 {\bf P}_{\frac{r}{\gamma_1}S}[g(\gamma_1)] \leq
\left({\bf P}_{\frac{r}{\gamma_1}S} \circ {\bf
P}_{\frac{\gamma_1}{\gamma_2}S}\right)[g(\gamma_2)]={\bf
P}_{\frac{r}{\gamma_2}S} [g(\gamma_2)]
\end{equation}
for $0< r<\gamma_1<\gamma_2<1$. Passing to the limit in
\eqref{PPPP2}, as $r\to 0$, and using the continuity of the
noncommutative Berezin transform, we deduce that
$$
\widetilde \tau [g(\gamma_1)]={\bf P}_0[g(\gamma_1)]\leq {\bf
P}_0[g(\gamma_2)]=\widetilde \tau [g(\gamma_2)].
$$
 Notice also that relation \eqref{gP}, implies
$$
g(0)\leq {\bf P}_0[g(\gamma_1)]=\widetilde \tau [g(\gamma_1)]\quad
\text{ for } 0<\gamma_1<1.
$$
Part (iii) is now obvious. This completes the proof.
\end{proof}

If $g$ is a   sub-pluriharmonic  curve  on $[0,1)$     and $h$ is
pluriharmonic  on the same interval such that $g(r)\leq h(r)$, $r\in
[0,1)$, we say that $h$ is a {\it pluriharmonic majorant} for $g$.
The next result provides a characterization of all sub-pluriharmonic
curves which admit free pluriharmonic majorants. In this case, we
find the least pluriharmonic majorant.

\begin{theorem}
\label{har-major}
 Let $g:[0,1)\to
\overline{\text{\boldmath{$\cP$}}_n^*+\text{\boldmath{$\cP$}}_n}^{\|\cdot\|}$
be a sub-pluriharmonic curve. Then there exists a pluriharmonic
majorant of $g$ if and only if
\begin{equation}
\label{supr}
 \sup_{0<r<1} \|\widetilde \tau [g(r)]\|<\infty,
\end{equation}
where $\widetilde \tau:=\tau\otimes \text{\rm id}$ and  $\tau$ is
the linear functional on $B(F^2(H_n))$ defined by
$\tau(f):=\left<f(1),1\right>$.  In this case, there is a least
pluriharmonic majorant for $g$, namely, the map $$ [0,1)\ni r\to
u(rS_1,\ldots, rS_n)\in
\overline{\text{\boldmath{$\cP$}}_n^*+\text{\boldmath{$\cP$}}_n}^{\|\cdot\|},
$$
where the free pluriharmonic function $u$ is given by
$$
u(X_1,\ldots, X_n):=\lim_{\gamma\to 1}{\bf
P}_{\frac{1}{\gamma}X}[g(\gamma)]
$$
for any $X:=(X_1,\ldots, X_n)\in [B(\cH)^n]_1$ and  the limit is in
the norm topology.
\end{theorem}
\begin{proof}
Assume that $u$ is a pluriharmonic majorant for $g$, i.e.,
\begin{equation*}   g(\gamma)\leq
u(\gamma)\quad \text{ for any }\  \gamma\in [0,1).
\end{equation*}
Since $\widetilde \tau$ is a positive map, we deduce that
$\widetilde \tau [g(\gamma)]\leq \widetilde \tau [u(\gamma)]$,
$\gamma\in [0,1)$. According to Theorem \ref{chara-ine}, we have
$$g(0)\leq \widetilde \tau [g(\gamma)]\qquad \text{ for } \ \gamma\in
(0,1).
$$
Since $u$ is a pluriharmonic function, we have $\widetilde \tau
[u(\gamma)]=u(0)$. Using these relations, we deduce that
$$
g(0)\leq \widetilde \tau [g(\gamma)]\leq  u(0)\quad \text{ for } \
\gamma\in (0,1).
$$
Taking into account that the operators $g(0)$, $u(0)$, and
$\widetilde \tau [g(\gamma)]$, $\gamma\in (0,1)$, are selfadjoint,
one can easily  obtain \eqref{supr}.

Conversely, assume that relation  \eqref{supr} holds. Define
$h_\gamma:[0,\gamma)\to
\overline{\text{\boldmath{$\cP$}}_n^*+\text{\boldmath{$\cP$}}_n}^{\|\cdot\|}$
by setting \begin{equation} \label{hga} h_\gamma(r):= {\bf
P}_{\frac{r}{\gamma}S}[g(\gamma)]\  \text{ for } 0\leq r<\gamma.
\end{equation}
Since $g$ is sub-pluriharmonic,  $h_\gamma$ is a pluriharmonic
majorant for $g$ on $[0,\gamma)$. Notice that if $f:[0,1)\to
\overline{\text{\boldmath{$\cP$}}_n^*+\text{\boldmath{$\cP$}}_n}^{\|\cdot\|}$
is continuous and pluriharmonic on $[0,\gamma]$ such that $f(r)\geq
g(r)$ for any $r\in [0,\gamma]$, then $f(r)\geq h_\gamma(r)$ for
$r\in [0,\gamma)$. Indeed, since $f(\gamma)\geq g(\gamma)$, the
Poisson transform   is a positive map, and $f$ is pluriharmonic, we
have
$$
f(r)= {\bf P}_{\frac{r}{\gamma}S}[f(\gamma)]\geq {\bf
P}_{\frac{r}{\gamma}S}[g(\gamma)]=h_\gamma(r)
$$
for $r\in [0,\gamma)$. This shows that $h_\gamma$ is the least
pluriharmonic majorant of $g$ on $[0,\gamma)$.

Now let $0<\gamma< \gamma'<1$.  Since $h_{\gamma'}$ is pluriharmonic
majorant for $g$ on $[0,\gamma')$, it is also a pluriharmonic
majorant for $g$ on $[0,\gamma)$. Due to  our result above, we have
$$
h_\gamma(r)\leq h_{\gamma'}(r)\quad \text{ for any } \
r\in[0,\gamma).
$$
Due to relation \eqref{hga}, we have $h_\gamma(0)=\widetilde \tau
[g(\gamma)]$ for $0<\gamma <1$ and, therefore,
$$
\sup_{0<\gamma<1}\|h_\gamma(0)\|=\sup_{0<\gamma<1}\|\widetilde \tau
[g(\gamma)]\|<\infty.
$$
Now, each $h_\gamma$ generates a unique pluriharmonic function on
$[B(\cH)^n]_\gamma$ by setting
\begin{equation}
\label{uga} u_\gamma(X):={\bf P}_{\frac{1}{t}X}[h_\gamma(t)]\quad
\text{ for }\ X\in [B(\cH)^n]_t  \text{ and }  t\in [0,\gamma).
\end{equation}
If $0<\gamma< \gamma'<1$, then $$u_\gamma(X)\leq u_{\gamma'}(X)\quad
\text{ for} \ X\in [B(\cH)^n]_\gamma.
$$
Since
$$\sup_{0<\gamma<1}\|h_\gamma(0)\|=\sup_{0<\gamma<1}\|u_\gamma(0)\|<\infty,
$$
we can use the Harnack type convergence theorem from
\cite{Po-pluriharmonic} to deduce the existence  of a pluriharmonic
function $u$ on $[B(\cH)^n]_1$ such that its radial function
satisfies  $u(r):=u(rS_1,\ldots, rS_n)=\lim_{\gamma\to 1}
u_\gamma(r)$ for any $r\in [0,1)$, where the convergence is in the
operator norm topology. On the other hand, using  relation
\eqref{uga}, the fact that $h_\gamma$ is pluriharmonic on
$[0,\gamma)$, and relation \eqref{hga}, we have
$$
u_\gamma (r)={\bf P}_{\frac{r}{t}S}[h_\gamma(t)]=h_\gamma(r)={\bf
P}_{\frac{r}{\gamma}S}[g(\gamma)]\quad \text{ for } \  0\leq
r<t<\gamma.
$$
Since $h_\gamma\geq g$ on $[0,\gamma)$ for each  $\gamma\in (0,1)$,
we deduce that $u\geq g$ on $[0,1)$. If $f$ is  any pluriharmonic
majorant for $g$ on $[0,1)$, then, as previously shown, we have
$f\geq h_\gamma$ on $[0,\gamma)$ for each $\gamma\in (0,1)$. Hence,
$f\geq u$ on $[0,1)$. Therefore, we have shown that
$u:[B(\cH)^n]_1\to B(\cE)\otimes_{min} B(\cH)$ defined by
\begin{equation*}
\begin{split}
u(X)&=\lim_{\gamma\to 1}{\bf P}_{\frac{1}{t}X}[h_\gamma(t)]=
\lim_{\gamma\to 1}{\bf P}_{\frac{1}{t}X}\left[{\bf
P}_{\frac{t}{\gamma}S}[g(\gamma)]\right]\\
&=\lim_{\gamma\to 1}{\bf P}_{\frac{1}{\gamma}X}[g(\gamma)]
\end{split}
\end{equation*}
for $X\in [B(\cH)^n]_t$ and $t\in (0,\gamma)$,   is a pluriharmonic
function on $[B(\cH)^n]_1$. Moreover, its radial function $
u(t)=\lim_{r\to 1}{\bf P}_{\frac{t}{r}S}[g(r)]$,  $ t\in [0,1)$, is
the least pluriharmonic majorant of $g$. This completes the proof.
\end{proof}

Now we can prove the following  maximum principle for
sub-pluriharmonic curves.

\begin{theorem}
\label{max-prin} If $g$  is a sub-pluriharmonic curve on $[0,1)$ and
$g(0)\geq g(r)$ for any $ r\in [0,1)$, then $g$ is a constant.
\end{theorem}
\begin{proof}
Since $\widetilde\tau$ is a positive map, we have $\widetilde \tau
[g(r)]\leq \widetilde \tau [g(0)]$ for  $ r\in [0,1)$. Hence and
using
 Corollary \ref{ine-subha}, we deduce that
 \begin{equation}
 \label{utata}
g(0)\leq \widetilde \tau [g(r)]\leq \widetilde \tau [g(0)]\quad
\text{ for } \ r\in (0,1).
 \end{equation}
On the other hand, since $g(0)$ is in
$\overline{\text{\boldmath{$\cP$}}_n^*+\text{\boldmath{$\cP$}}_n}^{\|\cdot\|}$,
so is $\varphi:=\widetilde \tau [g(0)]-g(0)$. Moreover, $\varphi$ is
positive with $\widetilde\tau (\varphi)=0$. Using Theorem 4.1 from
\cite{Po-pluriharmonic}, the pluriharmonic function  $h$ associated
with $\varphi$, i.e., $h(X)={\bf P}_X[\varphi],\quad X\in
[B(\cH)^n]_1, $ is positive and $h(0)=0$. Due to the maximum
principle for free pluriharmonic functions (see Theorem 2.9 from
\cite{Po-pluriharmonic}), we deduce that $h=0$. Since
$\varphi=\lim_{r\to 1} h(rS_1,\ldots, rS_n)$, we also have
$\varphi=0$, whence $g(0)=\widetilde \tau [g(0)]$. Due to  relation
\eqref{utata}, we have
\begin{equation}
\label{tau} \widetilde \tau [g(r)]=g(0)\quad \text{ for } \ r\in
[0,1).
\end{equation}
Hence and using the fact that $g(0)\geq g(r)$, $r\in [0,1)$, we
deduce that $\widetilde \tau [g(r)]-g(r)\geq 0$ for $r\in (0,1)$. A
similar argument as before implies $g(r)=\widetilde \tau [g(r)]$ for
$r\in (0,1)$.  Now taking into account \eqref{tau}, we get
$g(r)=g(0)$ for $r\in [0,1)$. Therefore, $\widetilde \tau
[g(r)]=\widetilde \tau [g(0)]=g(0)$ for $ r\in [0,1). $ This
completes the proof.
\end{proof}

A few remarks are necessary. We say that  a map
$\varphi:[0,\gamma)\to B(\cE)\otimes_{min} C^*(S_1,\ldots, S_n)$  is
a {\it $C^*$-harmonic   curve}  if it satisfies the Poisson mean
value property, i.e.,
\begin{equation*}
  \varphi(r)= {\bf P}_{\frac{r}{t} S}[\varphi(t)]\quad \text{
for } \ 0\leq r<t<\gamma.
\end{equation*}
 According to
\cite{Po-pluriharmonic}, there exists a one-to-one correspondence
$u\mapsto \varphi$ between the set of all  $C^*$-harmonic functions
on the noncommutative ball $[B(\cH)^n]_\gamma$ and the set of all
 $C^*$-harmonic  curves $\varphi:[0,\gamma)\to B(\cE)\otimes_{min} C^*(S_1,\ldots, S_n)$.
  We say that a map
$$\psi:[0,1)\to  B(\cE)\otimes_{min} C^*(S_1,\ldots, S_n)\quad
\text{ with }\ \psi(r)=\psi(r)^*, r\in [0,1), $$
 is a  {\it $C^*$-subharmonic   curve},
   provided that  for each $\gamma\in (0,1)$
 and each $C^*$-harmonic curve  $\varphi $ on the closed interval  on $[0,\gamma]$
 such that   $\varphi(r)=\varphi(r)^*$ for
$r\in [0,\gamma]$,
  if
 $\psi(\gamma)\leq \varphi(\gamma)$,   then
 $
\psi(r)\leq \varphi(r)\ \text{ for any } \ r\in [0,\gamma].
 $
 We remark that Theorem \ref{chara-ine} and Corollary
 \ref{ine-subha} have analogues for $C^*$-subharmonic curves. Since the proofs are
 similar, we shall omit them.
Notice also that any sub-pluriharmonic curve is $C^*$-subharmonic.

Finally, we mention that all the results of this section can be
written  for sub-pluriharmonic curves in the tensor algebra
$B(\cE)\otimes_{min} C^*(R_1,\ldots, R_n)$, where $R_1,\ldots, R_n$
are the right creation operators on the full Fock space.

\bigskip

\section{ Free pluriharmonic majorants and a
characterization of $H^2_{\bf ball} $ }

In this section we show that, for any
 free holomorphic function   $\Theta$ on $[B(\cH)^n]_1$ with coefficients in $B(\cE, \cY)$, the mapping
 $$
 \varphi:[0,1)\to B(\cE)\otimes_{min} C^*(R_1,\ldots, R_n),\quad
 \varphi(r)=\Theta(rR_1,\ldots, rR_n)^*\Theta(rR_1,\ldots, rR_n),
 $$
is a sub-pluriharmonic curve.  We prove that $\Theta$ is in the
Hardy space $H^2_{\bf ball}$  if and only if $\varphi$ has a
pluriharmonic majorant. In this case, we find Herglotz-Riesz and
Poisson  type representations for  the least pluriharmonic majorant
of $\varphi$.

First, we introduce some notation. Recall that $H_n$ is an
$n$-dimensional complex Hilbert space with orthonormal
      basis
      $e_1$, $e_2$, $\dots,e_n$, and
 the full Fock space  of $H_n$  is defined by
      $F^2(H_n):=\CC 1\oplus \bigoplus_{k\geq 1} H_n^{\otimes k}.$
Let $\FF_n^+$ be the unital free semigroup on $n$ generators
      $g_1,\dots,g_n$, and the identity $g_0$.
       We denote $e_\alpha:=
e_{i_1}\otimes\cdots \otimes  e_{i_k}$  if $\alpha=g_{i_1}\cdots
g_{i_k}$, where $i_1,\ldots, i_k\in \{1,\ldots,n\}$, and
$e_{g_0}:=1$. Note that $\{e_\alpha\}_{\alpha\in \FF_n^+}$ is an
orthonormal basis for $F^2(H_n)$.  We denote by $\tilde\alpha$ the
reverse of $\alpha\in \FF_n^+$, i.e.,
  $\tilde \alpha= g_{i_k}\cdots g_{i_k}$ if
   $\alpha=g_{i_1}\cdots g_{i_k}\in\FF_n^+$.

Let $\Theta:[B(\cH)^n]_1\to B(\cE,\cY) \otimes_{min} B(\cH)$ be a
free holomorphic function on $[B(\cH)^n]_1$ with coefficients in
$B(\cE,\cY)$ (in this case we denote $\Theta\in H_{\bf ball} (B(\cE,
\cY))$). Assume that $\Theta$ has the  representation
\begin{equation} \label{Th}
\Theta(X_1,\ldots, X_n):= \sum_{k=0}^\infty\sum_{|\alpha|=k}
A_{(\alpha)}\otimes X_\alpha.
\end{equation}
  We say that $\Theta$ is in the
noncommutative Hardy space $ H^2_{\bf ball} $ if there is a constant
$c>0$ such that $\sum_{\alpha\in \FF_n^+} A_{(\alpha)}^*
A_{(\alpha)}\leq cI_\cE$. When we want to emphasize  that the
coefficients of $\Theta$ are in $B(\cE, \cY)$, we denote $\Theta\in
H^2_{\bf ball}(B(\cE,\cY))$.  If we set $
\|\Theta\|_2:=\left\|\sum_{\alpha \in \FF_n^+} A_{(\alpha)}^*
A_{(\alpha)}\right\|^{1/2}<\infty, $ then $(H^2_{\bf ball},
\|\cdot\|_2)$ is a Banach space. We associate with each $\Theta \in
H^2_{\bf ball} $ the operator $\Gamma:\cE\to \cY\otimes F^2(H_n)$
defined by
\begin{equation}
\label{Ga}
 \Gamma x:=\sum_{\alpha\in \FF_n^+} A_{(\alpha)}x\otimes
e_{\tilde\alpha},\quad x\in \cE.
\end{equation}
We call $\Gamma$ the {\it symbol}  of $\Theta$.
 We will see later that $\Gamma x=\lim_{r\to 1}
\Theta(rR_1,\ldots, rR_n)(x\otimes 1)$, $x\in \cE$. Conversely, if
$\Gamma$ is an operator given by \eqref{Ga}, then relation
\eqref{Th} defines  a free holomorphic function  $\Theta$ in
$H^2_{\bf ball} $. Moreover, one can show  that $\Theta\in  H^2_{\bf
ball} $ and its symbol $\Gamma$ uniquely determine each other.

\begin{lemma}
\label{state} Let $\Theta$ be a free holomorphic function in
$H^2_{\bf ball}(B(\cE,\cY)) $  and let $\Gamma$ be its  symbol. Then
$\Theta$ has the state space realization
$$
\Theta(X_1,\ldots, X_n)= \left[E_{\cY}^*(I_\cY\otimes
P_{\CC})\otimes I_\cH\right]
\left[I_\cY\otimes\left(I_{F^2(H_n)\otimes \cH}-\sum_{i=1}^n
S_i^*\otimes X_i\right)^{-1}\right](\Gamma\otimes I_\cH)
$$
for $(X_1,\ldots, X_n)\in [B(\cH)^n]_1$, where $E_\cY:\cY\to
\cY\otimes F^2(H_n)$ is defined by setting  $E_\cY y=y\otimes 1$,
and $P_\CC$ denotes the orthogonal projection of $F^2(H_n)$ on
$\CC$.
\end{lemma}
\begin{proof}

Assume that  $\Theta$   has the representation $ \Theta(X_1,\ldots,
X_n):= \sum_{k=0}^\infty\sum_{|\alpha|=k} A_{(\alpha)}\otimes
X_\alpha$ for some coefficients $A_{(\alpha)}\in B(\cE, \cY)$, and
let $\Gamma:\cE\to \cY\otimes F^2(H_n)$  be its symbol  defined by
relation \eqref{Ga}. Notice that $$
A_{(\alpha)}=E_\cY^*(I_\cY\otimes P_\CC)(I_\cY\otimes S_{\tilde
\alpha}^*)\Gamma,\qquad \alpha\in \FF_n^+. $$ Therefore, we have
\begin{equation*}
\begin{split}
\Theta(X_1,\ldots, X_n)&= \sum_{k=0}^\infty\sum_{|\alpha|=k}
\left[E_\cY^*(I_\cY\otimes P_\CC)(I_\cY\otimes S_{\tilde
\alpha}^*)\Gamma\right]\otimes X_\alpha\\
&= \left[E_\cY^*(I_\cY\otimes P_\CC)\otimes I_\cH\right]
\left[\sum_{k=0}^\infty\sum_{|\alpha|=k}(I_\cY\otimes S_{\tilde
\alpha}^*)\otimes X_\alpha\right](\Gamma\otimes I_\cH)\\
&= \left[E_\cY^*(I_\cY\otimes P_\CC)\otimes I_\cH\right]
\left(I_{\cY\otimes F^2(H_n)\otimes \cH}- \sum_{i=1}^n I_\cY\otimes
S_i^*\otimes X_i\right)^{-1}(\Gamma\otimes I_\cH)
\end{split}
\end{equation*}
for  any $(X_1,\ldots, X_n)\in [B(\cH)^n]_1$. This completes the
proof.
\end{proof}

Now we introduce a large class of sub-pluriharmonic functions.

\begin{theorem}
\label{Th-sub}
 Let $\Theta $ be a free holomorphic function on $[B(\cH)^n]_1$ with
coefficients in $B(\cE,\cY)$. Then the map
$$
\varphi(r):=\Theta(rR_1,\ldots, rR_n)^*\Theta(rR_1,\ldots,
rR_n),\quad r\in [0,1),
$$
is a sub-pluriharmonic curve  in the tensor  algebra
$B(\cE)\otimes_{min}C^*(R_1,\ldots, R_n)$.
\end{theorem}
\begin{proof}

 First,  assume that   $ \Theta$  is  a free holomorphic function in
$H^2_{\bf ball} $ and let $\Gamma$ be  its symbol (see  \eqref{Ga}).
Define the free holomorphic function on $[B(\cH)^n]_1$   by setting
\begin{equation}
\begin{split}
\label{def-W} W(X_1,\ldots, X_n):= (\Gamma^*\otimes I_\cH)&
\left(I_{\cY\otimes F^2(H_n)\otimes \cH}+ \sum_{i=1}^n I_\cY\otimes
S_i^*\otimes X_i\right) \\
&\left(I_{\cY\otimes F^2(H_n)\otimes \cH}- \sum_{i=1}^n I_\cY\otimes
S_i^*\otimes X_i\right)^{-1}(\Gamma\otimes I_\cH).
\end{split}
\end{equation}

Consider the noncommutative Cauchy kernel
$$
\Phi(X_1,\ldots, X_n):= \left(I_{\cY\otimes F^2(H_n)\otimes \cH}-
\sum_{i=1}^n I_\cY\otimes S_i^*\otimes X_i\right)^{-1},\qquad
(X_1,\ldots, X_n)\in [B(\cH)^n]_1.
$$
Notice that $\Phi$ is a free holomorphic function on $[B(\cH)^n]_1$
and
\begin{equation}
\begin{split}
\label{phi-rel} \Phi(X_1,\ldots, X_n)&=I_{\cY\otimes F^2(H_n)\otimes
\cH}+\Phi(X_1,\ldots, X_n)\left(\sum_{i=1}^n I_\cY\otimes
S_i^*\otimes X_i\right)\\
& = I_{\cY\otimes F^2(H_n)\otimes \cH}+\left(\sum_{i=1}^n
I_\cY\otimes S_i^*\otimes X_i\right)\Phi(X_1,\ldots, X_n).
\end{split}
\end{equation}
A closer look at the definition of the free holomorphic function $W$
(see \eqref{def-W}) reveals that
\begin{equation}\label{W-def2}
\begin{split}
W(X_1,\ldots, X_n)= (\Gamma^*\Gamma\otimes I_\cH) +
2(\Gamma^*\otimes I_\cH) \left(  \sum_{i=1}^n I_\cY\otimes
S_i^*\otimes X_i\right) \Phi(X_1,\ldots, X_n)(\Gamma\otimes I_\cH).
\end{split}
\end{equation}
Now, we use  Lemma \ref{state} when $\cH:= F^2(H_n)$ and
$X_i:=rR_i$, $i=1,\ldots,n$ , for $r\in[0,1)$. Note that due to the
fact that $P_\CC=I_{F^2(H_n)}- S_1S_1^*-\cdots- S_nS_n^*$ and using
relation \eqref{phi-rel}, we deduce that
\begin{equation*}
\begin{split}
&\Theta(rR_1,\ldots, rR_n)^*\Theta(rR_1,\ldots, rR_n) \\
&\quad= (\Gamma^*\otimes I_{F^2(H_n)}) \left( I_{\cY\otimes
F^2(H_n)\otimes
F^2(H_n)}-\sum_{i=1}^n I_\cY\otimes S_i\otimes rR_i^*\right)^{-1}\\
 &\qquad\qquad\qquad \times(I_\cY\otimes P_\CC\otimes I_{F^2(H_n)})  \left(
I_{\cY\otimes F^2(H_n)\otimes F^2(H_n)}-\sum_{i=1}^n I_\cY\otimes
S_i^*\otimes rR_i\right)^{-1}\\
&\quad= (\Gamma^*\otimes I_{F^2(H_n)}) \Phi(rR_1,\ldots, rR_n) ^*
\left[I_\cY\otimes \left(I_{F^2(H_n)}-\sum_{i=1}^n
S_iS_i^*\right)\otimes I_{F^2(H_n)}\right]\\
&\qquad\qquad\qquad \times \Phi(rR_1,\ldots, rR_n)
(\Gamma\otimes I_{F^2(H_n)})\\
&\quad= (\Gamma^*\otimes I_{F^2(H_n)}) \Phi(rR_1,\ldots, rR_n) ^*
   \Phi(rR_1,\ldots, rR_n)(\Gamma\otimes I_{F^2(H_n)})\\
   &\qquad -
   (\Gamma^*\otimes I_{F^2(H_n)}) \Phi(rR_1,\ldots, rR_n) ^*
\left[I_\cY\otimes \left(\sum_{i=1}^n S_iS_i^*\right)\otimes
I_{F^2(H_n)}\right] \Phi(rR_1,\ldots, rR_n)
(\Gamma\otimes I_{F^2(H_n)})\\
&\quad=(\Gamma^*\otimes I_{F^2(H_n)})
 \left[I_{\cY\otimes F^2(H_n)\otimes
F^2(H_n)}+\left(\sum_{i=1}^n I_\cY\otimes S_i\otimes
rR_i^*\right)\Phi(rR_1,\ldots, rR_n)^*\right]\\
&\qquad\qquad\qquad \times \left[I_{\cY\otimes F^2(H_n)\otimes
F^2(H_n)}+\Phi(rR_1,\ldots, rR_n) \left(\sum_{i=1}^n I_\cY\otimes
S_i\otimes rR_i^*\right)\right](\Gamma\otimes I_{F^2(H_n)})\\
&\qquad - (\Gamma^*\otimes I_{F^2(H_n)}) \Phi(rR_1,\ldots, rR_n) ^*
\left[I_\cY\otimes \left(\sum_{i=1}^n S_iS_i^*\right)\otimes
I_{F^2(H_n)}\right] \Phi(rR_1,\ldots, rR_n)
(\Gamma\otimes I_{F^2(H_n)})\\
&\quad= (\Gamma^*\Gamma \otimes I_{F^2(H_n)})+(\Gamma^*\otimes
I_{F^2(H_n)}) \left(\sum_{i=1}^n I_\cY\otimes S_i\otimes
rR_i^*\right)\Phi(rR_1,\ldots, rR_n) ^* (\Gamma\otimes
I_{F^2(H_n)})\\
&\qquad + (\Gamma^*\otimes I_{F^2(H_n)}) \Phi(rR_1,\ldots, rR_n)
\left(\sum_{i=1}^n I_\cY\otimes S_i^*\otimes rR_i\right)
(\Gamma\otimes
I_{F^2(H_n)})\\
&\qquad + (\Gamma^*\otimes I_{F^2(H_n)}) \Phi(rR_1,\ldots,
rR_n)^* \left(\sum_{j=1}^n I_\cY\otimes S_j\otimes rR_j^*\right)\\
&\qquad \qquad \qquad \times \left(\sum_{i=1}^n I_\cY\otimes
S_i^*\otimes rR_i\right) \Phi(rR_1,\ldots,
rR_n)(\Gamma\otimes I_{F^2(H_n)})\\
&\qquad - (\Gamma^*\otimes I_{F^2(H_n)}) \Phi(rR_1,\ldots, rR_n)^*
\left(\sum_{j=1}^n I_\cY\otimes S_jS_j^*\otimes I_{F^2(H_n)}\right)
  \Phi(rR_1,\ldots,
rR_n)(\Gamma\otimes I_{F^2(H_n)}).
\end{split}
\end{equation*}

Hence and using relation \eqref{W-def2},  we deduce that
\begin{equation}\label{TT=VV}
\begin{split}
&\Theta(rR_1,\ldots, rR_n)^*\Theta(rR_1,\ldots, rR_n)\\
&\quad = \frac{1}{2}\left[ W(rR_1,\ldots, rR_n)+W(rR_1,\ldots,
rR_n)^*\right] -(1-r^2)(\Gamma^*\otimes I_{F^2(H_n)})
\Phi(rR_1,\ldots, rR_n)^* \\
&\qquad \qquad\qquad \qquad \times\left(\sum_{j=1}^n I_\cY\otimes
S_jS_j^*\otimes I_{F^2(H_n)}\right)
  \Phi(rR_1,\ldots,
rR_n)(\Gamma\otimes I_{F^2(H_n)})
\end{split}
\end{equation}
for any $r\in [0,1)$. Consequently, we have
\begin{equation}
\label{ine-imp}
 \Theta(rR_1,\ldots, rR_n)^*\Theta(rR_1,\ldots, rR_n)
\leq \frac{1}{2}\left[ W(rR_1,\ldots, rR_n)+W(rR_1,\ldots,
rR_n)^*\right]
\end{equation}
for any $r\in[0,1)$, which proves that $W$ is a free holomorphic
function with positive real part.

For each $s\in[0,1)$, define the operator $\Lambda_s:\cY\otimes
F^2(H_n)\to \cY\otimes F^2(H_n)$ by setting
$$\Lambda_s\left(\sum_{\alpha\in \FF_n^+}h_\alpha\otimes e_\alpha\right)
:=\sum_{\alpha\in \FF_n^+} h_\alpha\otimes s^{|\alpha|} e_\alpha,
\qquad \sum_{\alpha\in \FF_n^+}\| h_\alpha\|^2<\infty.
$$
It is easy to see that $\Lambda_s$ is a positive operator such that
$\|\Lambda_s\|\leq 1$  and  $\lim_{s\to 1}\Lambda_s= I$ in the
strong operator topology. Note also that
\begin{equation}
\label{Lam-s} \Lambda_s(I_\cY\otimes S_i)=s(I_\cY\otimes S_i)
\Lambda_s,\qquad i=1,\ldots,n.
\end{equation}
Fix $s\in [0,1)$ and set $\Theta_s(X_1,\ldots,
X_n):=\Theta(sX_1,\ldots, sX_n)$ for $X_1,\ldots, X_n)\in
[B(\cH)^n]_1$. It clear that $\Theta_s$ is free holomorphic on an
open neighborhood of $[B(\cH)^n]_1$ and, therefore, continuous on
the closed ball $[B(\cH)^n]_1^-$. Let $\Gamma_s$ be  the symbol
operator associated with $\Theta_s$ (see relation \eqref{Ga}) and
let $W_s$ be the operator associated to $\Gamma_s$ (by relation
\eqref{W-def2}). Then we have $\Gamma_s=\Lambda_s \Gamma$, where
$\Gamma$ is the  symbol associated with $\Theta$. Due to relations
\eqref{W-def2} and \eqref{Lam-s}, we deduce that
\begin{equation*}
\begin{split}
W_s(X_1,\ldots, X_n)&=
 (\Gamma^*\Lambda_s^2\Gamma\otimes I_\cH) +
2(\Gamma^* \Lambda_s\otimes I_\cH) \left(  \sum_{i=1}^n I_\cY\otimes
S_i^*\otimes X_i\right) \Phi(X_1,\ldots, X_n)(\Lambda_s\Gamma\otimes
I_\cH)\\
&= (\Gamma^*\Lambda_s^2\Gamma\otimes I_\cH) + 2(\Gamma^*
\Lambda_s^2\otimes I_\cH) \left(  \sum_{i=1}^n I_\cY\otimes
S_i^*\otimes sX_i\right) \Phi(sX_1,\ldots, sX_n)(\Gamma\otimes
I_\cH).
\end{split}
\end{equation*}
Consequently, $W_s$ is a free holomorphic function on an open set
containing the closed ball $[B(\cH)^n]_1^-$, and
\begin{equation} \label{Ws-lim}
\text{\rm SOT-} \lim_{s\to1} W_s= W.
\end{equation}
Now, from the first part of the proof (see \eqref{TT=VV}), we have
\begin{equation*}
\begin{split}
\text{\rm Re}\,&W_s(rR_1,\ldots, rR_n)-\Theta_s(rR_1,\ldots, rR_n)^* \Theta_s(rR_1,\ldots, rR_n)\\
&=(1-r^2)(\Gamma_s^*\otimes I_{F^2(H_n)}) \Phi(rR_1,\ldots, rR_n)^*
\left(\sum_{j=1}^n I_\cY\otimes r^2S_jS_j^*\otimes
I_{F^2(H_n)}\right)\\
&\qquad \qquad \times
  \Phi(rR_1,\ldots,
rR_n)(\Gamma_s\otimes I_{F^2(H_n)})\\
&=(1-r^2)(\Gamma^*\otimes I_{F^2(H_n)}) \Phi(srR_1,\ldots,
srR_n)^*(\Lambda_s\otimes I_{F^2(H_n)})\\
&\qquad \qquad \times \left(\sum_{j=1}^n I_\cY\otimes
r^2S_jS_j^*\otimes I_{F^2(H_n)}\right)(\Lambda_s\otimes
I_{F^2(H_n)}) \Phi(srR_1,\ldots, srR_n)(\Gamma\otimes I_{F^2(H_n)})
\end{split}
\end{equation*}
for any $r\in [0,1)$.  Hence we deduce that
\begin{equation}
\label{Re-ine-s} \text{\rm Re}\,W_s(rR_1,\ldots, rR_n)\geq
\Theta_s(rR_1,\ldots, rR_n)^* \Theta_s(rR_1,\ldots, rR_n)
\end{equation}
for any $r\in [0,1)$. Moreover,
 taking into account the continuity (in the norm operator topology) of the free holomorphic functions
 $\Theta_s$ and  $W_s$
 on the closed ball $[B(\cH)^n]_1^-$,   the continuity of  the map
 $$[0,1]\ni r\mapsto  \Phi(srR_1,\ldots, srR_n),$$
  we deduce that
 \begin{equation}
 \label{W-s}
 \begin{split}
\text{\rm Re}\,W_s(R_1,\ldots, R_n)=\Theta_s(R_1,\ldots, R_n)^*
\Theta_s(R_1,\ldots, R_n)
 \end{split}
 \end{equation}
for any $s\in [0,1)$.

Now, we prove that the map $\varphi$ is a sub-pluriharmonic curve.
Fix $s\in (0,1)$ and assume that $u$ is a free holomorphic function
on $[B(\cH)^n]_s$ and continuous on the closed ball $[B(\cH)^n]_s^-$
such that
\begin{equation}
\label{TTR}
 \Theta(sR_1,\ldots, sR_n)^*\Theta(sR_1,\ldots, sR_n)\leq
\text{\rm Re}\, u(sR_1,\ldots, sR_n).
\end{equation}
Let $t\in (s,1)$ and  set $\Theta':=\Theta_t$. Since $\Theta'$ is in
$H^2_{\bf ball} $,  let $W'$ be the free holomorphic function
associated with $\Theta'$ (according to relation \ref{def-W}). Now,
we can apply the results  above (see \eqref{W-s}) to $\Theta'$ and
$W'$ and obtain
\begin{equation}
\label{W'} \text{\rm Re}\,W'_\tau(R_1,\ldots,
R_n)=\Theta'_\tau(R_1,\ldots, R_n)^* \Theta'_\tau(R_1,\ldots, R_n)
\end{equation}
for any $\tau\in [0,1)$. Using \eqref{W'} and \eqref{TTR}, we have
\begin{equation*}
\begin{split}
\text{\rm Re}\,W'_{\frac{s}{t}}(R_1,\ldots,
R_n)&=\Theta'_{\frac{s}{t}}(R_1,\ldots, R_n)^*
\Theta'_{\frac{s}{t}}(R_1,\ldots, R_n)\\
&=\Theta(sR_1,\ldots, sR_n)^*\Theta(sR_1,\ldots, sR_n)\\
&\leq \text{\rm Re}\, u(sR_1,\ldots, sR_n).
\end{split}
\end{equation*}
  Employing  the noncommutative Poisson transform, we deduce that
\begin{equation*}
 \begin{split}
\text{\rm Re}\,W'_{\frac{s}{t}}(\gamma R_1,\ldots,\gamma R_n)&= {\bf
P}_{\gamma
R}\left[\text{\rm Re}\,W'_{\frac{s}{t}}( R_1,\ldots, R_n)\right]\\
&\leq  {\bf P}_{\gamma R}\left[ \text{\rm Re}\, u( sR_1,\ldots, sR_n)\right]\\
&\leq \text{\rm Re}\, u(\gamma sR_1,\ldots, \gamma sR_n)
\end{split}
 \end{equation*}
for any $\gamma\in [0,1)$.  Consequently,  applying  inequality
\eqref{Re-ine-s} to $\Theta'$ and $W'$, we have
$$
\Theta'_{\frac{s}{t}}(\gamma R_1,\ldots, \gamma R_n)^*
\Theta'_{\frac{s}{t}}(\gamma R_1,\ldots, \gamma R_n)\leq \text{\rm
Re}\,W'_{\frac{s}{t}}(\gamma R_1,\ldots, \gamma R_n)\leq \text{\rm
Re}\, u(\gamma sR_1,\ldots, \gamma sR_n)
$$
for any $\gamma\in [0,1)$.  Hence, we deduce that
$$
\Theta(s\gamma R_1,\ldots, s\gamma R_n)^* \Theta(s\gamma R_1,\ldots,
s\gamma R_n)\leq \text{\rm Re}\, u(s\gamma R_1,\ldots, s\gamma R_n)
$$
for any $\gamma\in [0,1)$. Therefore,
$$
\Theta(t R_1,\ldots, t R_n)^* \Theta(t R_1,\ldots, t R_n)\leq
\text{\rm Re}\, u(t R_1,\ldots, t R_n)
$$
for any $t\in [0,s]$. This proves that  $\varphi$ is a
sub-pluriharmonic curve. The proof is complete.
\end{proof}

 Let  $\cP^{(m)}$, $m=0,1,\ldots$,
be the set of all polynomials of degree $\leq m$   in $e_1,\ldots,
e_n$, i.e.,
$$
\cP^{(m)}:=\text{ \rm span} \{ e_\alpha: \ \alpha\in \FF_n^+,
|\alpha|\leq m\}\subset F^2(H_n),
$$
and define the nilpotent operators $R_i^{(m)}: \cP^{(m)}\to
\cP^{(m)}$ by
$$
R_i^{(m)}:=P_{\cP^{(m)}} R_i |_{\cP^{(m)}},\quad i=1,\ldots, n,
$$
where $R_1,\ldots, R_n$ are the right creation operators on the Fock
space $F^2(H_n)$ and $P_{\cP^{(m)}}$ is the orthogonal projection of
$F^2(H_n)$ onto $\cP^{(m)}$. Notice that $R_\alpha^{(m)}=0$ if
$|\alpha|\geq m+1$.

 We can provide now a characterization of the noncommutative Hardy
 space $H^2_{\bf ball} $ in terms of pluriharmonic majorants.

\begin{theorem}
\label{Th-sub2}
 Let $\Theta $ be a free holomorphic function on $[B(\cH)^n]_1$ with
coefficients in $B(\cE,\cY)$. Then  $ \Theta$ is in $H^2_{\bf ball}
$ if and only if  the map $\varphi$ defined by
$$
\varphi(r):=\Theta(rR_1,\ldots, rR_n)^*\Theta(rR_1,\ldots,
rR_n),\quad r\in [0,1),
$$
has a pluriharmonic majorant.   In this case, the least
pluriharmonic majorant $\psi$  for $\varphi$ is given by
\begin{equation}
\label{==}
 \psi(r):=\text{\rm Re}\, W(rR_1, \ldots rR_n),\quad r\in[0,1),
\end{equation}
  where $W$ is the free holomorphic function  having the
  Herglotz-Riesz
  type representation
  \begin{equation}\label{Herg}
  W(X_1,\ldots,
  X_n)=({\mu}_\theta\otimes \text{\rm id})\left[\left(I+\sum_{i=1}^n R_i^*\otimes
  X_i\right)\left(I-\sum_{i=1}^n R_i^*\otimes
  X_i\right)^{-1}\right]
  \end{equation}
for $(X_1,\ldots, X_n)\in [B(\cH)^n]_1$, where
$\mu_\theta:\cR_n^*+\cR_n\to B(\cE)$ is the completely positive
linear map uniquely determined by   the equation
\begin{equation}
\label{def-mu} \left<\mu_\theta(R_{\widetilde \alpha}^*)x,y\right>:=
\lim_{r\to 1}\left<\Theta(rR_1,\ldots, rR_n)^*(I_\cY\otimes
S_{\widetilde \alpha}^*)\Theta(rR_1,\ldots, rR_n)(x\otimes 1),
(y\otimes 1)\right>
\end{equation}
for $\alpha\in \FF_n^+$ and $x,y\in \cE$.
\end{theorem}
\begin{proof}  According to Theorem
\ref{Th-sub}, $\varphi$ is a  sub-pluriharmonic curve in
$B(\cE)\otimes_{min}C^*(R_1,\ldots, R_n)$. Assume that $\varphi$ has
a pluriharmonic majorant. Let $u$ be a free pluriharmonic function
on $[B(\cH)^n]_1$ with coefficients in $B(\cE)$ such that
$\varphi(r)\leq u(rR_1, \ldots, rR_n)$ for any $r\in [0,1)$. Let
$\Theta$ have the representation \begin{equation} \label{TH-again}
 \Theta(X_1,\ldots,
X_n):= \sum_{k=0}^\infty\sum_{|\alpha|=k} A_{(\alpha)}\otimes
X_\alpha,\qquad (X_1,\ldots, X_n)\in [B(\cH)^n]_1,
\end{equation}
and let $u$ have the representation
$$
u(X_1,\ldots, X_n)=\sum_{k=1}^\infty \sum_{|\alpha|=k}
 B_{(\alpha)}^*\otimes  X_\alpha^*+ \sum_{k=0}^\infty \sum_{|\alpha|=k}
 B_{(\alpha)}\otimes  X_\alpha,\qquad (X_1,\ldots, X_n)\in [B(\cH)^n]_1.
$$
Notice that $B_{(0)}\geq 0$ and, for each $h\in \cE$ and $r\in
[0,1)$, we have
\begin{equation*}
\begin{split}
\sum_{\alpha\in \FF_n^+} r^{2|\alpha|} \|A_{(\alpha)}h\|^2
&=\|\Theta(rR_1,\ldots, rR_n) (h\otimes 1)\|^2\\
&=\left<\varphi(r) (h\otimes 1), h\otimes 1\right> \leq \left<
u(rR_1, \ldots, rR_n)(h\otimes 1), h\otimes
1\right>\\
&=\left<u(0)(h\otimes 1), h\otimes 1\right> =\left<
B_{(0)}h,h\right>.
\end{split}
\end{equation*}
 Hence, we deduce that there is a constant $c>0$ such that
$\sum_{\alpha\in \FF_n^+} A_{(\alpha)}^* A_{(\alpha)}\leq cI_\cE$.
This shows that $ \Theta$ is in $H^2_{\bf ball} $.

 Conversely, assume that  $ \Theta$ is in $H^2_{\bf ball} $.  A
 closer look at the proof of Theorem \ref{Th-sub} (see
relation \eqref{ine-imp}) shows that  $\text{\rm Re}\, W$  is a
pluriharmonic
 majorant for $\varphi$, where $W$ is given by relation
 \eqref{def-W}.
It remains to show that  $\text{\rm Re}\, W$  is the least
pluriharmonic
 majorant for $\varphi$ and satisfies relation \eqref{Herg}.
 Let $G$ be a free holomorphic function on
$[B(\cH)^n]_1$  with coefficients in $B(\cE)$ such that $\text{\rm
Re}\, G\geq 0$ and
$$
\Theta(t R_1,\ldots, t R_n)^* \Theta(t R_1,\ldots, t R_n)\leq
\text{\rm Re}\, G(t R_1,\ldots, t R_n)
$$
for any $t\in [0,1)$. If  $s\in (0,1)$, then we have
$$
\Theta_s(R_1,\ldots,  R_n)^*\Theta_s(R_1,\ldots,  R_n)\leq \text{\rm
Re}\, G(s R_1,\ldots, s R_n).
$$
Using \eqref{W-s}, we deduce that
$$
\text{\rm Re}\, W_s( R_1,\ldots,  R_n)\leq \text{\rm Re}\, G(s
R_1,\ldots, s R_n).
$$
Taking the compression to  $\cE\otimes \cP^{(m)}$ (see the
definition preceding this theorem), we obtain
$$
\text{\rm Re}\, W_s( R_1^{(m)},\ldots,  R_n^{(m)})\leq \text{\rm
Re}\, G(s R_1^{(m)},\ldots, s R_n^{(m)}).
$$
Now, using relation \eqref{Ws-lim} and taking $s\to 1$, we obtain
$$
\text{\rm Re}\, W( R_1^{(m)},\ldots,  R_n^{(m)})\leq \text{\rm Re}\,
G( R_1^{(m)},\ldots,  R_n^{(m)})
$$
for any $m=0,1,\ldots $. According to \cite{Po-pluriharmonic}, we
deduce that $\text{\rm Re}\, W\leq \text{\rm Re}\, G$, which shows
that the map $ \psi(r):=\text{\rm Re}\, W(rR_1, \ldots rR_n)$, \
$r\in [0,1),$ is  the least pluriharmonic majorant for $\varphi$.

Notice that, due to relation  \eqref{W-def2}, we have
$$
W(X_1,\ldots, X_n) = \sum_{k=0}^\infty\sum_{|\alpha|=k}
C_{(\alpha)}\otimes X_\alpha,\qquad (X_1,\ldots, X_n)\in
[B(\cH)^n]_1,
$$
where
\begin{equation}
\label{Calpha} C_{(0)}=\Gamma^*\Gamma\quad \text{ and } \
C_{(\alpha)}=2\Gamma^*(I_\cY\otimes S_{\widetilde\alpha}^*)\Gamma
\quad \text{ if }\ |\alpha|\geq 1.
\end{equation}
Using the definition of $\Gamma$ and the  representation
\eqref{TH-again}, we deduce that, for any $\alpha\in \FF_n^+$ and
$x,y\in \cE$,
\begin{equation*}
\begin{split}
\left<\Gamma^*(I_\cY\otimes S_{\widetilde\alpha}^*)\Gamma x,y\right>
&= \left< \sum_{\beta\in \FF_n} A_{(\beta)}x\otimes e_{\tilde
\beta},(I_\cY\otimes S_{\widetilde\alpha}^*)\left(\sum_{\gamma\in
\FF_n} A_{(\gamma)}x\otimes e_{\tilde \gamma}\right)\right>\\
&= \sum_{\beta, \gamma\in \FF_n^+}\left<A_{(\beta)} x,
A_{(\gamma)}y\right>\left<e_{\tilde \beta}, e_{\tilde \alpha\tilde
\gamma}\right>\\
&=\sum_{ \gamma\in \FF_n^+}\left<A_{(\gamma\alpha)} x,
A_{(\gamma)}y\right>.
\end{split}
\end{equation*}
Therefore $\Gamma^*(I_\cY\otimes
S_{\widetilde\alpha}^*)\Gamma=\sum_{\gamma\in \FF_n^+} A_{(\gamma)}
A_{(\gamma\alpha)}$, $\alpha\in \FF_n^+$, where the convergence is
in the week operator topology.

On the other hand, notice that the limit in \eqref{def-mu} exists.
Indeed, since $\Theta$ is in the Hardy space $H^2_{\bf ball} $, we
have $\sum_{\beta\in \FF_n^+} \|A_\beta x\|^2\leq c\|x\|^2$ for some
constant $c>0$. This implies
$$
\lim_{r\to 1}\sum_{\beta\in \FF_n} r^{|\beta|}A_{(\beta)}x\otimes
e_{\tilde \beta}=\sum_{\beta\in \FF_n} A_{(\beta)}x\otimes e_{\tilde
\beta},
$$
whence
\begin{equation*}
\begin{split}
\lim_{r\to 1}&\left<\Theta(rR_1,\ldots, rR_n)^*(I_\cY\otimes
S_{\widetilde \alpha}^*)\Theta(rR_1,\ldots, rR_n)(x\otimes 1),
(y\otimes 1)\right>\\
&= \left< \sum_{\beta\in \FF_n}r^{|\beta|} A_{(\beta)}x\otimes
e_{\tilde \beta},(I_\cY\otimes
S_{\widetilde\alpha}^*)\left(\sum_{\gamma\in
\FF_n} r^{|\gamma|}A_{(\gamma)}x\otimes e_{\tilde \gamma}\right)\right>\\
&= \left< \sum_{\beta\in \FF_n}  A_{(\beta)}x\otimes e_{\tilde
\beta},(I_\cY\otimes S_{\widetilde\alpha}^*)\left(\sum_{\gamma\in
\FF_n}  A_{(\gamma)}x\otimes e_{\tilde \gamma}\right)\right>\\
&= \left<\Gamma^*(I_\cY\otimes S_{\widetilde\alpha}^*)\Gamma
x,y\right>.
\end{split}
\end{equation*}
Therefore $\mu_\theta(R_{\widetilde \alpha}^*)=\Gamma^*(I_\cY\otimes
S_{\widetilde\alpha}^*)\Gamma$  and $\mu_\theta(R_{\widetilde
\alpha})=\Gamma^*(I_\cY\otimes S_{\widetilde\alpha})\Gamma$ for any
$\alpha\in \FF_n^+$. Consequently,
$$\mu_\theta(p(R_1,\ldots, R_n))=
\Gamma^*[I_\cY\otimes  p(S_1,\ldots, S_n)]\Gamma
$$
 for any
polynomial $p(R_1,\ldots, R_n)=\sum_{|\alpha|\leq m} (b_\alpha
R_\alpha^*+ a_\alpha R_\alpha)$ in $\cR_n^*+\cR_n$. This implies
that $\mu_\theta$ has a unique extension to a completely positive
linear map on  $\cR_n^*+\cR_n$. According  to relation
\eqref{Calpha}, we have
$$
C_{(0)}=\mu_\theta(I) \quad \text{ and } \ C_{(\alpha)}=2\mu_\theta
(R_{\widetilde \alpha}^*)\quad \text{ if }\ |\alpha|\geq 1.
$$
Hence and due to the fact that $\mu_\theta$ is a bounded linear map,
we have
\begin{equation*}
\begin{split}
\widetilde{\mu}_\theta&\left[\left(I+\sum_{i=1}^n R_i^*\otimes
  X_i\right)\left(I-\sum_{i=1}^n R_i^*\otimes
  X_i\right)^{-1}\right]\\
&\qquad =
\widetilde{\mu}_\theta\left[I+2\sum_{k=1}^\infty\sum_{|\alpha|=k}
R_{\tilde\alpha}^*\otimes X_\alpha\right]\\
&\qquad =\mu_\theta(I)+2\sum_{k=1}^\infty\sum_{|\alpha|=k}
\mu(R_{\tilde\alpha}^*)\otimes X_\alpha\\
&\qquad =\sum_{k=0}^\infty\sum_{|\alpha|=k} C_{(\alpha)}\otimes
X_\alpha=W(X_1,\ldots, X_n),
\end{split}
\end{equation*}
which proves relation \eqref{Herg}. This completes the proof.
\end{proof}

We recall \cite{Po-pluriharmonic} that, in the particular case  when
 $f=I$, the identity on $F^2(H_n)$, and  $\mu$ is a bounded linear
functional on
 $C^*(R_1,\ldots, R_n)$, then the Berezin transform
$B_\mu(I, \cdot \,)$ coincides with the noncommutative Poisson
transform $\cP\mu:[B(\cH)^n]_1\to B(\cH)$  associated with $\mu$,
i.e.,
$$
(\cP\mu)(X_1,\ldots, X_n):=(\mu\otimes \text{\rm
id})\left[P(R,X)\right], \qquad X:=(X_1,\ldots,X_n)\in [B(\cH)^n]_1,
$$
where the Poisson kernel is given by
$$
P(R,X):=\sum_{k=1}^\infty\sum_{|\alpha=k}
R_{\widetilde\alpha}\otimes X_\alpha^*
+I+\sum_{k=1}^\infty\sum_{|\alpha=k} R_{\widetilde\alpha}^*\otimes
X_\alpha
$$
and the series are convergent in the operator norm topology. In the
particular case when $n=1$, $\cH=\CC$, $X=re^{i\theta}\in \DD$, and
$\mu$ is a complex Borel measure on $\TT$, the Poisson transform
$\cP\mu$ can be identified with the classical Poisson transform of
$\mu$, i.e.,
$$
\frac{1}{2\pi}\int_{-\pi}^{\pi} P_r(\theta-t)d\mu(t),
$$
where $P_r(\theta-t):=\frac{1-r^2}{1-2r\cos(\theta-t)+r^2}$ is the
Poisson kernel.

Using now Theorem \ref{har-major} and Theorem \ref{Th-sub2},  we can
deduce the following result.

\begin{corollary}\label{new-cond} Let  $ \Theta$ be  in
$H^2_{\bf ball}$ and have    the representation
\begin{equation*}
 \Theta(X_1,\ldots,
X_n):= \sum_{k=0}^\infty\sum_{|\alpha|=k} A_{(\alpha)}\otimes
X_\alpha,\qquad (X_1,\ldots, X_n)\in [B(\cH)^n]_1.
\end{equation*}
 Under the conditions of Theorem \ref{Th-sub2},  the  least
 pluriharmonic majorant $\text{\rm Re}\, W$  satisfies the relations
\begin{equation*}
\begin{split}
\text{\rm Re}\, W(X_1,\ldots, X_n)&=\lim_{t\to 1}{\bf
P}_{\frac{1}{t}X} [\Theta(tR_1,\ldots, tR_n)^*\Theta(tR_1,\ldots,
tR_n)]\\
&= (\cP\mu_\theta)( X_1,\ldots, X_n)\\
&=\sum_{k=1}^\infty \sum_{|\alpha|=k} D_{(\alpha)}^*\otimes
X_\alpha^* +\sum_{k=0}^\infty \sum_{|\alpha|=k} D_{(\alpha)}\otimes
X_\alpha
\end{split}
\end{equation*}
for any $X:=(X_1,\ldots, X_n)\in [B(\cH)^n]_1$, where $\cP
\mu_\theta$ is the noncommutative Poisson transform of $\mu_\theta$
   and $$ D_{(\alpha)}=\sum_{\gamma\in \FF_n^+} A_{(\gamma)}
A_{(\gamma\alpha)},\quad \alpha\in \FF_n^+,
$$
where the convergence is in the week operator topology.
\end{corollary}

\bigskip

\section{ The unit ball of $H^2_{\bf ball}$, geometric structure, and representations  }

In this section we obtain a characterization of  the unit ball of
$H^2_{\bf ball}$ and  provide a parametrization and concrete
representations of all pluriharmonic majorants  for   the
sub-pluriharmonic curve
$$ \varphi(r):=\Theta(rR_1,\ldots,
rR_n)^*\Theta(rR_1,\ldots, rR_n),\qquad r\in [0,1),$$ where  $
\Theta$ is in the unit ball of  $H^2_{\bf ball}$.

We need a few notations. As in the previous section, we denote by
$H_{\text{\bf ball}}(B(\cE,\cY))$ the set of all free holomorphic
functions on the noncommutative ball $[B(\cH)^n]_1$ and coefficients
in $B(\cE,\cY)$. Let $H_{\text{\bf ball}}^\infty (B(\cE,\cY))$
denote the set of all elements $F$ in  $H_{\text{\rm
ball}}(B(\cE,\cY))$  such that
$$
\|F\|_\infty:=\sup  \|F(X_1,\ldots, X_n)\|<\infty,
$$
where the supremum is taken over all $n$-tuples  of operators
$(X_1,\ldots, X_n)\in [B(\cH)^n]_1$, where $\cH$ is an infinite
dimensional   Hilbert space. Denote by $H^+_{\bf ball}(B(\cE))$ the
set of all free holomorphic functions $f$   on the noncommutative
ball $[B(\cH)^n]_1$ with coefficients in $B(\cE)$, where $\cE$ is a
separable Hilbert space, such that $\text{\rm Re}\, f\geq 0$.
Consider  the following sets:
\begin{equation*}
\begin{split}
{\bf H}_1^+(B(\cE)) &:=\left\{ f\in H^+_{\bf ball}(B(\cE)) :\
f(0)=I \right\} \text{ and } \\
{\bf H}_ 0^\infty (B(\cE))&:=\left\{ g\in H_{\bf ball}^\infty(B(\cE)
): \ g(0)=0\right\}.
\end{split}
\end{equation*}
According to \cite{Po-free-hol-interp}, the {\it noncommutative
Cayley transform} is a bijection
$$\text{\boldmath{$\cC$}}:  {\bf H}_1^+(B(\cE))\to
\left[ {\bf H}_ 0^\infty (B(\cE))\right]_{\leq 1}\quad \text{
defined by } \ \text{\boldmath{$\cC$}} [f]:=g,$$
 where $g\in  \left[ {\bf H}_ 0^\infty (B(\cE))\right]_{\leq 1}$ is
  uniquely determined by
the formal power series
  $( \widetilde
 f-1)(1+\widetilde f)^{-1}$, where
$\widetilde f$ is the power series associated with $f$. In this
case, we have
$$\text{\boldmath{$\cC$}}^{-1}[G](X)=[I+G(X)][I-G(X)]^{-1}, \qquad X\in [B(\cH)^n]_1.
$$

We recall that if $T:\cH\to \cH'$ is a contraction, then
$D_T:=(I-T^*T)^{1/2}$  and $\cD_T:=\overline{D_T\cH}$.
 The first result of this section provides a  parametrization
  for the pluriharmonic majorants of $\Theta^* \Theta$.

\begin{theorem}
\label{parametriz} Let  $ \Theta$  be a free holomorphic function in
the unit ball of  $H^2_{\bf ball}( B(\cE,\cY))$, and let $F$ be  a
free holomorphic function  in $H_{\bf ball}(B(\cE))$   such that
$F(0)=I$. Then
\begin{equation}
\label{T<RE} \Theta(rR_1,\ldots, rR_n)^*\Theta(rR_1,\ldots,
rR_n)\leq \text{\rm Re}\, F(rR_1,\ldots, rR_n) \quad \text{ for } \
r\in [0,1),
\end{equation}
if and only if there exists $G$ in the unit ball  of  ${\bf H}_
0^\infty ( B(\cD_\Gamma))$ \, such that
\begin{equation}
\label{formula} F(X)=W(X) + \left(D_\Gamma \otimes
I\right)[I+G(X)][I-G(X)]^{-1}\left(D_\Gamma \otimes I\right),\quad
X\in [B(\cH)^n]_1,
\end{equation}
 where $\Gamma$  is the symbol of  $\Theta$ and $W$ is defined by relation
 \eqref{def-W}. Moreover, $F$ and $G$ uniquely
determine each other.
\end{theorem}

\begin{proof} Let $G$ be  in the unit ball of  \,  ${\bf H}_ 0^\infty ( B(\cD_\Gamma))$
 and define $F$
by relation \eqref{formula}. Due to \eqref{W-def2}, we have
$W(0)=\Gamma^* \Gamma \otimes I_\cH$. Since $ G(0)=I$, we deduce
that
\begin{equation*}
\begin{split}
F(0)&=W(0)+ \left(D_\Gamma \otimes I_\cH\right)
[I+G(0)][I-G(0)]^{-1}\left(D_\Gamma \otimes I_\cH\right)\\
&=(\Gamma^* \Gamma +D_\Gamma^2)\otimes I_\cH= I.
\end{split}
\end{equation*}
According to Theorem \ref{Th-sub} (see \eqref{ine-imp}), we have
\begin{equation}
\Theta(rR_1,\ldots, rR_n)^*\Theta(rR_1,\ldots, rR_n)\leq \text{\rm
Re}\, W(rR_1,\ldots, rR_n) \quad \text{ for } \ r\in [0,1).
\end{equation}
Due to the properties of the noncommutative Cayley transform, we
have
$$\text{\rm Re}\, \text{\boldmath{$\cC$}}^{-1}[G](rR_1,\ldots, rR_n)
\geq 0\quad \text{ for any } \ r\in [0,1).
$$ Consequently, we have
\begin{equation*}
\begin{split}
\text{\rm Re}\,F(rR_1,\ldots, rR_n)&=\text{\rm Re}\,W(rR_1,\ldots,
rR_n) + \left(D_\Gamma \otimes I\right)[\text{\rm
Re}\,\text{\boldmath{$\cC$}}^{-1}[G](rR_1,\ldots,
rR_n)]\left(D_\Gamma \otimes I\right)\\
&\geq \Theta(rR_1,\ldots, rR_n)^*\Theta(rR_1,\ldots, rR_n)
\end{split}
\end{equation*}
for any   $ r\in [0,1)$. Therefore, $F$ satisfies inequality
\eqref{T<RE}.

Conversely,   assume that $F$ is a free holomorphic function  with
$F(0)=I$ and satisfying relation \eqref{T<RE}. According to Theorem
\ref{Th-sub2},
$$\text{\rm Re}\,F(rR_1,\ldots, rR_n)\geq \text{\rm Re}\,W(rR_1,\ldots, rR_n)\quad
\text{ for any } \ r\in[0,1).
$$
Hence, $\Psi:=F-W$ is a  free holomorphic function with positive
real part and
$$\Psi(0)=F(0)-W(0)=(I-\Gamma^*\Gamma)\otimes
I_\cH=D_\Gamma ^2\otimes I_\cH.
$$
We claim that $\Psi$ has a unique factorization of the form
\begin{equation}
\label{factori} \Psi(X)=(D_\Gamma\otimes I_\cH)
\Lambda(X)(D_\Gamma\otimes I_\cH),\quad X\in [B(\cH)^n]_1,
\end{equation}
where $\Lambda$ is a free holomorphic function  with coefficients in
$B(\cD_\Gamma)$ such that $\Lambda(0)=I$ and $\text{\rm Re}\,\Lambda
\geq 0$.
To this end, assume that  $\Psi$ has the representation
$\Psi(X)=\sum_{\alpha \in \FF_n^+} Q_{(\alpha)}\otimes X_\alpha$,
 $Q_{(\alpha)}\in B(\cE)$, with $Q_{(0)}=D_\Gamma^2$.  For each $k\geq 1$, consider the
subspace $\cM:=\text{\rm span}\,\{1, e_\alpha:\ \alpha\in \FF_n^+
\text{ and } |\alpha|=k\}$. Notice that, for each $r\in [0,1)$, the
positive  operator
$$P_{\cE\otimes \cM} [\Psi(rR_1,\ldots, rR_n)^*
+\Psi(rR_1,\ldots, rR_n)]|_{\cE\otimes \cM}
$$
has the operator   matrix representation

$$M(r):= \left[\begin{matrix}
2D_\Gamma^2 &[r^{|\alpha|}Q_{(\alpha)}:\ |\alpha|=k]\\
\left[\begin{matrix}  r^{|\alpha|}Q_{(\alpha)}\\
:\\
|\alpha|=k
\end{matrix}
\right]& \left[\begin{matrix}
 2D_\Gamma^2&\cdots& 0\\
 \vdots&\ddots&\vdots\\
0&\cdots & 2D_\Gamma^2
\end{matrix}\right]
\end{matrix}
 \right],
$$
where $[r^{|\alpha|}Q_{(\alpha)}:\ |\alpha|=k]$ is the row operator
with  the entries  $r^{|\alpha|}Q_{(\alpha)}$,  $|\alpha|=k$. The
notation for the column operator is now obvious. Taking $r\to 1$, we
deduce that $M:=M(1)\geq 0$.

We recall (see \cite{GGK}) that  a block operator matrix
$\left[\begin{matrix} A&B\\
B^*&C \end{matrix}\right]$, where $A\in B(\cH)$, $B\in B(\cG, \cH)$,
and $C\in B(\cG)$, is positive if and only if $A$ and $C$ are
positive and there exists a contraction $T:\overline{C\cG}\to
\overline{A\cH}$ satisfying $B^*=C^{1/2} T^* A^{1/2}$. Applying this
result to the operator $M$, we find  a contraction
$$
\left[\begin{matrix}   Z_{(\alpha)}\\
:\\
|\alpha|=k
\end{matrix}
\right] :\cD_\Gamma\to \oplus_{|\alpha|=k} \cD_\Gamma
$$
such that
$$
\left[\begin{matrix}   Q_{(\alpha)}\\
:\\
|\alpha|=k
\end{matrix}
\right] =
 \left[\begin{matrix}
 2D_\Gamma&\cdots& 0\\
 \vdots&\ddots&\vdots\\
0&\cdots & 2D_\Gamma
\end{matrix}
 \right]
 \left[\begin{matrix}   Z_{(\alpha)}\\
:\\
|\alpha|=k
\end{matrix}
\right] D_\Gamma
$$
for each $k\geq 1$. Hence, we have $Q_{(\alpha)}=D_\Gamma
K_{(\alpha)} D_\Gamma$, where $K_{(0)}=I$ and
$K_{(\alpha)}:=2Z_{(\alpha)}$ for any $\alpha\in \FF_n^+$ with
$|\alpha|=k\geq 1$. Since $\sum_{|\alpha|=k} Z_{(\alpha)}^*
Z_{(\alpha)}\leq I$, we deduce that $\lim_{k\to\infty}
\left\|\sum_{|\alpha|=k} K_{(\alpha)}^*
K_{(\alpha)}\right\|^{\frac{1} {2k}} \leq 1$. This implies (see
\cite{Po-holomorphic}) that
$$\Lambda(X_1,\ldots, X_n)=\sum_{k=0} \sum_{|\alpha|=k}
K_{(\alpha)}\otimes X_\alpha,\quad (X_1,\ldots, X_n)\in
[B(\cH)^n]_1,
$$
is a free holomorphic function on $[B(\cH)^n]_1$. Using  the
relations above, we deduce the factorization
$\Psi(X)=(D_\Gamma\otimes I_\cH) \Lambda(X)(D_\Gamma\otimes
I_\cH),\quad X\in [B(\cH)^n]_1.$

Now, since  the operator $\Lambda(rR_1,\ldots, rR_n)$,  $r\in
[0,1)$, is acting on the Hilbert space $\cD_\Gamma\otimes F^2(H_n)$
and
\begin{equation*}
\begin{split}
\Psi(rR_1,\ldots, rR_n)^* &+\Psi(rR_1,\ldots, rR_n)\\
&= (D_\Gamma\otimes I_{F^2(H_n)})[\Lambda(rR_1,\ldots, rR_n)^*
+\Lambda(rR_1,\ldots, rR_n)](D_\Gamma\otimes I_{F^2(H_n)}),
\end{split}
\end{equation*}
we deduce that $\Lambda(rR_1,\ldots, rR_n)^* +\Lambda(rR_1,\ldots,
rR_n)\geq 0$ for any $r\in [0,1)$. Using the noncommutative Poisson
transform, we have $\text{\rm Re}\, \Lambda\geq 0$, which proves our
claim. Due to the properties of the Cayley transform,
$G:=\text{\boldmath{$\cC$}}[\Lambda]$ is in  the unit ball of
$H_0^\infty(B(\cD_\Gamma))$. Finally,  since $\Psi:=F-W$ and  using
the factorization \eqref{factori}, we deduce \eqref{formula}. The
fact that $F$ and $G$ uniquely determine each other is  now obvious
since the Cayley transform is a bijection. This completes the proof.
\end{proof}

\begin{corollary}
Under the conditions of Theorem $\ref{parametriz}$, there is only
one $F$ satisfying \eqref{T<RE} if and only if $\Gamma$ is an
isometry. In this case, $F=W$.
\end{corollary}

According to \cite{Po-holomorphic} and \cite{Po-pluriharmonic}, the
noncommutative Hardy space
  $H_{\text{\rm ball}}^\infty (B(\cE,\cY))$   can be identified to the operator
  space
$ B(\cE,\cY)\bar\otimes F_n^\infty$ (the weakly closed operator
space generated by the spatial tensor product), where $F_n^\infty$
is the noncommutative analytic Toeplitz algebra (see \cite{Po-von},
\cite{Po-multi}, \cite{Po-analytic}). We consider the noncommutative
Schur class
$$
\cS_{\text{\bf ball}}(B(\cE,\cY)):=\left\{G\in H_{\text{\bf
ball}}^\infty (B(\cE,\cY)):\ \|G\|_\infty\leq 1\right\},
$$
which can be identified to the unit ball of  the operator
  space
$ B(\cE,\cY)\bar\otimes F_n^\infty$. We also use the notation
$$\cS^0_{\text{\bf ball}}(B(\cE,\cY)):=\{G\in \cS_{\text{\bf
ball}}(B(\cE,\cY)):\ G(0)=0\}.
$$

In what follows we use the notation $\cE^{(n)}$ for the direct sum
of $n$-copies of the  Hilbert space $\cE$. The next results provides
a Schur type representation for the unit ball of $H^2_{\bf ball}$.

\begin{theorem}
\label{one-to-one} Let $\Theta \in H_{\text{\bf ball}}^2
(B(\cE,\cY))$ be such that $\|\Theta\|_2\leq 1$, and let
$\Gamma_\theta:\cE\to \cY\otimes F^2(H_n)$ be its symbol. Then there
is a one-to-one correspondence $J_\Theta$ between the noncommutative
Schur class $\cS_{\text{\bf
ball}}(B(\cD_{\Gamma_\theta},\cD_{\Gamma_\theta}^{(n)}))$ and the
set of all function matrices
$\left[\begin{matrix} L\\
M_1\\
\vdots\\
M_n\end{matrix}\right]$  in  the noncommutative   Schur class \,
$\cS_{\text{\bf ball}}(B(\cE,\cY\oplus \cE^{(n)}))$ satisfying  the
equation
\begin{equation}
\label{Th-rep} \Theta(X)=L(X)\left[ I_{\cE\otimes \cH}-
\sum_{i=1}^n(I_\cE\otimes X_i)M_i(X)\right]^{-1}\quad  \text{ for
  } \ X:=(X_1,\ldots, X_n)\in [B(\cH)^n]_1.
\end{equation}
More precisely,  the map
$$
J_\Theta: \cS_{\text{\bf
ball}}(B(\cD_{\Gamma_\theta},\cD_{\Gamma_\theta}^{(n)}))\to
\cS_{\text{\bf ball}}(B(\cE,\cY\oplus \cE^{(n)}))
$$
is defined by setting
$$
J_\Theta \left[\begin{matrix}
\varphi_1\\
\vdots\\
\varphi_n\end{matrix}\right]=\left[\begin{matrix} L\\
M_1\\
\vdots\\
M_n\end{matrix}\right],
$$
where $L\in H_{\text{\bf ball}} (B(\cE,\cY))$ is  given  by
\begin{equation}
\label{L-def}
 L(X)=2\Theta(X)[F(X)+I]^{-1}, \quad X\in [B(\cH)^n]_1,
\end{equation}
 the free holomorphic function $F\in H_{\text{\bf ball}} (B(\cE ))$ is  defined  by
 \begin{equation}\label{F-def}
 \begin{split}
 F(X)&:= (\Gamma_\theta^*\otimes I_\cH)
\left(I_{\cY\otimes F^2(H_n)\otimes \cH}+ \sum_{i=1}^n I_\cY\otimes
S_i^*\otimes X_i\right) \\
&\qquad \times \left(I_{\cY\otimes F^2(H_n)\otimes \cH}-
\sum_{i=1}^n I_\cY\otimes
S_i^*\otimes X_i\right)^{-1}(\Gamma_\theta\otimes I_\cH)\\
 &\qquad +(D_{\Gamma_\theta}\otimes I_\cH)
 \left[ I+\sum_{i=1}^n (I_\cE\otimes X_i)\varphi_i(X)\right]
 \left[ I-\sum_{i=1}^n (I_\cE\otimes X_i)\varphi_i(X)\right]^{-1}(D_{\Gamma_\theta}\otimes
 I_\cH),
 \end{split}
 \end{equation}
and $M_1,\ldots, M_n\in H_{\text{\bf ball}} (B(\cE ))$ are uniquely
determined by the equation
\begin{equation}
\label{M-def}(I_\cE\otimes X_1)M_1+\cdots +(I_\cE\otimes X_n)M_n=
[F(X)-I][F(X)+ I]^{-1}, \quad X:=(X_1,\ldots, X_n)\in [B(\cH)]_1.
\end{equation}
In particular, the representation \eqref{Th-rep} is unique if and
only if $\Gamma_\theta$ is an isometry.
\end{theorem}

\begin{proof}
Assume that $\Theta\in H_{\text{\bf ball}} (B(\cE,\cY))$ and
$\|\Theta\|_2\leq 1$. Consider $\Phi:=\left[\begin{matrix}
\varphi_1\\
\vdots\\
\varphi_n\end{matrix}\right]$   in the noncommutative Schur class
$\cS_{\text{\bf
ball}}(B(\cD_{\Gamma_\theta},\cD_{\Gamma_\theta}^{(n)}))$. It is
easy to see  that $ \Phi\in  \cS_{\text{\bf
ball}}(B(\cD_{\Gamma_\theta},\cD_{\Gamma_\theta}^{(n)}))$ if and
only if the function
$$
\chi(X):=(I_\cE\otimes X_1)\varphi_1(X)+\cdots +(I_\cE\otimes
X_n)\varphi_n(X),\quad X:=(X_1,\ldots, X_n)\in [B(\cH)^n]_1,
$$
is in $\cS^0_{\text{\bf
ball}}(B(\cD_{\Gamma_\theta},\cD_{\Gamma_\theta} ))$. Since $$\left[
I+\sum_{i=1}^n (I_\cE\otimes X_i)\varphi_i(X)\right]
 \left[ I-\sum_{i=1}^n (I_\cE\otimes X_i)\varphi_i(X)\right]^{-1}$$
 is the noncommutative Cayley transform of $\chi$, it makes sense to
define $F$ by relation \eqref{F-def}. According to Theorem
\ref{parametriz}, $F\in H_{\bf ball}(B(\cE)$ has the properties
$F(0)=I$,  $\text{\rm Re}\, F\geq 0$, and
\begin{equation}
\label{Th-ine-Re} \Theta(rR_1,\ldots, rR_n)^*\Theta(rR_1,\ldots,
rR_n)\leq \text{\rm Re}\, F(rR_1,\ldots, rR_n) \quad \text{ for } \
r\in [0,1).
\end{equation}
Since $F\in {\bf H}_1^+(B(\cE))$, the noncommutative Cayley
transform of $F$, i.e.,
$\text{\boldmath{$\cC$}}[F]:=(F-I)(F+I)^{-1}$, is in ${\bf
H}_0^\infty(B(\cE))$ and $\|\text{\boldmath{$\cC$}}[F]\|_\infty\leq
1$. Consequently, there are some unique $M_1,\ldots, M_n\in H_{\bf
ball}(B(\cE))$ such that
$$\text{\boldmath{$\cC$}}[F](X)=(I_\cE\otimes X_1)M_1(X)+\cdots
+(I_\cE\otimes X_n)M_n(X),\quad X:=(X_1,\ldots, X_n)\in [B(\cH)]_1.
$$
Since  $\|\text{\boldmath{$\cC$}}[F]\|_\infty\leq 1$, we deduce that
$\left\|\sum_{i=1}^n (I_\cE\otimes rS_i)M_i(rS_1,\ldots,
rS_n)\right\|\leq 1$ for $r\in [0,1)$. Taking into account that
$S_1,\ldots, S_n$ are isometries with orthogonal ranges, we obtain
the inequality $$\sum_{i=1}^n r^2 M_i(rS_1,\ldots, rS_n)^*
M_i(rS_1,\ldots, rS_n)\leq I. $$ Hence, we deduce that
$$
\left\|\left[\begin{matrix}
M_1\\
\vdots\\
M_n\end{matrix}\right] \right\|_\infty=\lim_{r\to 1}
\left\|\left[\begin{matrix}
M_1(rS_1,\ldots, rS_n)\\
\vdots\\
M_n(rS_1,\ldots, rS_n)\end{matrix}\right] \right\|\leq 1,
$$
which shows that $\left[\begin{matrix}
M_1\\
\vdots\\
M_n\end{matrix}\right]$ in is the noncommutative   Schur class \,
$\cS_{\text{\bf ball}}(B(\cE,  \cE^{(n)}))$. Now,  consider the free
holomorphic function
\begin{equation}
\label{Psi-def} \Psi(X):= (I_\cE\otimes X_1)M_1(X)+\cdots
+(I_\cE\otimes X_n)M_n(X),\quad X:=(X_1,\ldots, X_n)\in [B(\cH)]_1,
\end{equation}
and  notice that
\begin{equation}
\label{I-P}
\begin{split}
 I-\Psi&=I-\text{\boldmath{$\cC$}}[F](X)\\
&=[(F+I)-(F-I)](F+I)^{-1}\\
&=2(F+I)^{-1}.
\end{split}
\end{equation}
Hence,  and  defining  $L$ by relation \eqref{L-def}, we deduce that
$L=\Theta (I- \Psi)$.   Consequently, $L\in H_{\bf
ball}(B(\cE,\cY))$ and
$$
\Theta(X)=L(X)\left[ I_{\cE\otimes \cH}- \sum_{i=1}^n(I_\cE\otimes
X_i)M_i(X)\right]^{-1}\quad  \text{ for
  } \ X:=(X_1,\ldots, X_n)\in [B(\cH)^n]_1.
  $$
Therefore, relation \eqref{Th-rep} holds. Now we show that $\left[\begin{matrix} L\\
M_1\\
\vdots\\
M_n\end{matrix}\right]$ is in the Schur class $\cS_{\text{\bf
ball}}(B(\cE,\cY\oplus \cE^{(n)}))$. Since $F$ is the inverse Cayley
transform of $\Psi$, i.e., $F=(I+\Psi)(I-\Psi)^{-1}$,  for any $r\in
[0,1)$, we have

\begin{equation*}
\begin{split}
\text{\rm Re}\,& F(rR_1, \ldots rR_n)\\
&= \frac{1}{2}\left[(I-\Psi(rR_1, \ldots rR_n)^*)^{-1}(I+\Psi(rR_1,
\ldots rR_n)^*+(I+\Psi(rR_1, \ldots rR_n))(I-\Psi(rR_1, \ldots
rR_n))^{-1}\right]\\
&=\frac{1}{2}(I-\Psi(rR_1, \ldots rR_n)^*)^{-1}\left[(I+\Psi(rR_1,
\ldots rR_n)^*)(I-\Psi(rR_1, \ldots
rR_n))\right.\\
&\qquad \qquad +\left.(I-\Psi(rR_1, \ldots rR_n)^*)(I+\Psi(rR_1,
\ldots rR_n))\right](I-\Psi(rR_1, \ldots
rR_n))^{-1}\\
&=(I-\Psi(rR_1, \ldots rR_n)^*)^{-1}[I-\Psi(rR_1, \ldots
rR_n)^*\Psi(rR_1, \ldots rR_n)](I-\Psi(rR_1, \ldots rR_n))^{-1}.
\end{split}
\end{equation*}
 Hence, and using relations \eqref{Th-rep}, \eqref{Th-ine-Re},  and \eqref{Psi-def},  we
obtain
\begin{equation*}
\begin{split}
 &(I-\Psi(rR_1,
\ldots rR_n)^*)^{-1} L(rR_1,\ldots, rR_n)^* L(rR_1,\ldots,
rR_n)(I-\Psi(rR_1, \ldots rR_n))^{-1}\\
&\qquad =\Theta(rR_1,\ldots, rR_n)^*\Theta(rR_1,\ldots, rR_n)\\
&\qquad \leq \text{\rm Re}\, F(rR_1,\ldots, rR_n)  \\
&\qquad=(I-\Psi(rR_1, \ldots rR_n)^*)^{-1}[(I-\Psi(rR_1, \ldots
rR_n)^*\Psi(rR_1, \ldots rR_n)](I-\Psi(rR_1, \ldots rR_n))^{-1}
\end{split}
\end{equation*}
for any $r\in [0,1)$. Consequently, we deduce that
$$
\Psi(rR_1, \ldots rR_n)^*\Psi(rR_1, \ldots rR_n)+L(rR_1,\ldots,
rR_n)^* L(rR_1,\ldots, rR_n)\leq I
$$
for any $r\in [0,1)$. Due to relation \eqref{Psi-def}, we have
$$
\Psi(rR_1, \ldots rR_n)^*\Psi(rR_1, \ldots rR_n)=\sum_{i=1}^n r^2
M_i(rR_1,\ldots, rR_n)^*M_i(rR_1,\ldots, rR_n)\quad \text{ for  }\
r\in [0,1).
$$
Combining these relations, we deduce that
\begin{equation}
\label{MiL} r^2\sum_{i=1}^n\|M_i(rR_1,\ldots,
rR_n)x\|^2+\|L(rR_1,\ldots, rR_n)x\|^2\leq \|x\|^2
\end{equation}
for any $x\in \cE\otimes F^2(H_n)$ and any $r\in [0,1)$.
 Remember that
$H_{\text{\rm ball}}^\infty (B(\cN,\cM))$   can be identified to the
operator
  space
$ B(\cN,\cM)\bar\otimes R_n^\infty$ for any Hilbert spaces $\cN$ and
$\cM$. In particular, since $L$ and $M_i$ are bounded free
holomorphic functions,  according to \cite{Po-pluriharmonic}, there
exist $\widetilde M_i \in B(\cE)\bar\otimes \cR_n^\infty$ and
$\widetilde L\in B(\cE,\cY)\bar\otimes\cR_n^\infty$ such that
\begin{equation*}
\begin{split}
\widetilde M_i&=\text{\rm SOT-}\lim_{r\to 1}M_i(rR_1,\ldots,
rR_n) \ \text{ and } \\
\widetilde L&=\text{\rm SOT-}\lim_{r\to 1}L(rR_1,\ldots, rR_n).
\end{split}
\end{equation*}
Passing to the limit in \eqref{MiL}, we obtain
$
\sum_{i=1}^n\|\widetilde M_i x\|^2+\|\widetilde Lx\|^2\leq \|x\|^2
$
for any $x\in \cE\otimes F^2(H_n)$. Therefore,
$\left[\begin{matrix} \widetilde L\\
\widetilde M_1\\
\vdots\\
\widetilde M_n \end{matrix}\right]$ is a contraction in
$B(\cE,\cY\oplus \cE^{(n)})\bar\otimes \cR_n^\infty$,
 which shows that
$\left[\begin{matrix} L\\
M_1\\
\vdots\\
M_n\end{matrix}\right]$ is in the noncommutative  Schur class
$\cS_{\text{\bf ball}}(B(\cE,\cY\oplus \cE^{(n)}))$.
We remark that $\left[\begin{matrix}
\varphi_1\\
\vdots\\
\varphi_n\end{matrix}\right]$ and $F$ uniquely determine each other
by Theorem \ref{parametriz}. On the other hand, $F$ and
 $
\left[\begin{matrix}
M_1\\
\vdots\\
M_n\end{matrix}\right]$ uniquely determine each other via the
noncommutative Cayley transform. Therefore, $J_\theta$ is a
one-to-one mapping. Now, let us prove  the surjectivity of  the
correspondence $J_\theta$.
Let $\left[\begin{matrix} L\\
M_1\\
\vdots\\
M_n\end{matrix}\right]$ be in the Schur class $\cS_{\text{\bf
ball}}(B(\cE,\cY\oplus \cE^{(n)}))$ such that relation
\eqref{Th-rep} holds. Therefore, we have
\begin{equation}
\label{L-sum}
 L(rR_1,\ldots, rR_n)^* L(rR_1,\ldots, rR_n)
+\sum_{i=1}^n M_i(rR_1,\ldots, rR_n)^* M_i(rR_1,\ldots, rR_n)\leq I.
\end{equation}
Hence $\left[\begin{matrix}
M_1\\
\vdots\\
M_n\end{matrix}\right]$ is in   $\cS_{\text{\bf ball}}(B(\cE,
\cE^{(n)}))$, which implies that  the free holomorphic function
$$\Psi(X):= (I_\cE\otimes X_1)M_1(X)+\cdots +(I_\cE\otimes
X_n)M_n(X),\quad X:=(X_1,\ldots, X_n)\in [B(\cH)]_1,
$$
is in $\cS^0_{\text{\bf ball}}(B(\cE))$. Let $F$ be the inverse
Cayley transform of $\Psi$. Then   $\text{\rm Re}\, F\geq 0$,
$F(0)=I$, and, due to  relations \eqref{Th-rep} and \eqref{L-sum},
we have
\begin{equation*}
\begin{split}
 &\Theta(rR_1,\ldots, rR_n)^*\Theta(rR_1,\ldots, rR_n)\\
&\qquad=(I-\Psi(rR_1, \ldots rR_n)^*)^{-1} L(rR_1,\ldots, rR_n)^*
L(rR_1,\ldots, rR_n)(I-\Psi(rR_1, \ldots rR_n))^{-1}\\
&\qquad\leq (I-\Psi(rR_1, \ldots rR_n)^*)^{-1}[(I-\Psi(rR_1, \ldots
rR_n)^*\Psi(rR_1, \ldots rR_n)](I-\Psi(rR_1, \ldots rR_n))^{-1}\\
&\qquad=\text{\rm Re}\, F(rR_1, \ldots rR_n)
\end{split}
\end{equation*}
for any $r\in [0,1)$. The latter equality was proved before, using
the fact that $F$ is the inverse Cayley transform of $\Psi$. In
particular, since $F(0)=I$, we  can use the inequality above to
deduce that
$$
\|\Theta(rR_1,\ldots, rR_n)(h\otimes 1)\|^2\leq \left<F(0)(h\otimes
1), h\otimes 1\right>=\|h\|^2
$$
for any $h\in\cE$ and $r\in [0,1)$.  Hence $\|\Theta\|_2\leq 1$.
Moreover,  since
$$\Theta(rR_1,\ldots, rR_n)^*\Theta(rR_1,\ldots,
rR_n)\leq \text{\rm Re}\, F(rR_1, \ldots rR_n),\qquad r\in [0,1),
$$
we can apply Theorem \ref{parametriz} to show that $F$ has the form
\eqref{F-def} for some $\left[\begin{matrix}
\varphi_1\\
\vdots\\
\varphi_n\end{matrix}\right]$ in $\cS_{\text{\bf
ball}}(B(\cD_{\Gamma_\theta},\cD_{\Gamma_\theta}^{(n)}))$. Now,
since $F$ is the inverse Cayley transform of  $\Psi$ we have
$\Psi=(F-I)(F+I)^{-1}$. Notice also that due to relations
\eqref{Th-rep} and \eqref{I-P}, we deduce that
$$L(X)=2\Theta(X)[F(X)+I]^{-1},\qquad X\in [B(\cH)^n]_1.
$$
Therefore, $ J_\Theta \left[\begin{matrix}
\varphi_1\\
\vdots\\
\varphi_n\end{matrix}\right]=\left[\begin{matrix} L\\
M_1\\
\vdots\\
M_n\end{matrix}\right] $. Finally, the representation \eqref{Th-rep}
is unique if and only if $D_{\Gamma_\theta}=0$, i.e.,
$\Gamma_\theta$ is an isometry. The proof is complete.
\end{proof}

 A closer look at the proof of Theorem \ref{one-to-one} reveals that
 we have also proved the following result.

\begin{corollary}
\label{caract-2} Let $\Theta:[B(\cH)^n]_1\to B(\cE,\cY)\otimes_{min}
B(\cH)$ be a   function. Then  $\Theta$ is  a free holomorphic
function in  $H_{\text{\bf ball}}^2 (B(\cE,\cY))$ with
$\|\Theta\|_2\leq 1$ if and only if it has a representation
\begin{equation*}
  \Theta(X)=L(X)\left[ I_{\cE\otimes \cH}-
\sum_{i=1}^n(I_\cE\otimes X_i)M_i(X)\right]^{-1}\quad  \text{ for
  } \ X:=(X_1,\ldots, X_n)\in [B(\cH)^n]_1,
\end{equation*}
 where
$\left[\begin{matrix} L\\
M_1\\
\vdots\\
M_n\end{matrix}\right]$   is  in the noncommutative   Schur class \,
$\cS_{\text{\bf ball}}(B(\cE,\cY\oplus \cE^{(n)}))$.
\end{corollary}

\bigskip

 \section{  All solutions  of the generalized  noncommutative commutant
 lifting problem   }

We introduce   a generalized noncommutative  commutant lifting ({\bf
GNCL})  problem, which extends to a multivariable setting  several
lifting problems including the classical Sz.-Nagy--Foia\c s
commutant lifting and the extensions obtained by Treil-Volberg,
Foia\c s-Frazho-Kaashoek, and Biswas-Foia\c s-Frazho, as well as
  their multivariable noncommutative versions.

We  solve the {\bf GNCL}  problem  and, using the results regarding
sub-pluriharmonic functions
 and  free pluriharmonic majorants on noncommutative balls,
 we provide a complete description of all solutions. In particular,
 we obtain   a concrete  Schur type  description of all solutions in  the noncommutative commutant lifting
 theorem.

An $n$-tuple  $T:=(T_1,\dots, T_n)$ of bounded linear  operators
acting on a common Hilbert space $\cH$
   is called
contractive (or row contraction) if
$$
T_1T_1^*+\cdots +T_nT_n^*\leq I_\cH.
$$
The defect operators associated with $T$  are
$$
\Delta_{T^*}:=\left(I_\cH-\sum_{i=1}^n T_iT_i^*\right)^{1/2}\in
B(\cH)\quad \text{ and }\quad
\Delta_{T}:=\left(\left[\delta_{ij}I_\cH-T_i^*T_j\right]_{n\times
n}\right)^{1/2}\in B(\cH^{(n)}),
$$
while the defect spaces  of $T$ are
$\cD_*=\cD_{T^*}:=\overline{\Delta_{T^*}\cH}$ and
$\cD=\cD_{T}:=\overline{\Delta_{T}\cH^{(n)}}$, where
$\cH^{(n)}:=\oplus_{i=1}^n \cH$ denotes the direct sum of $n$ copies
of $\cH$.
We say that an $n$-tuple  $V:=(V_1,\dots, V_n)$ of isometries on a
Hilbert space $\cK\supset \cH$ is a  minimal isometric dilation of
$T$ if the following properties are satisfied:
\begin{enumerate}
\item[(i)]  $V_1V_1^*+\cdots +V_nV_n^*\le I_\cK;$
\item[(ii)]  $V_i^*|_\cH=T_i^*, \ i=1,\dots,n;$
\item[(iii)] $\cK=\bigvee_{\alpha\in \FF^+_n} V_\alpha \cH.$
\end{enumerate}
The isometric dilation theorem for row contractions
\cite{Po-isometric}
 asserts that
  every  row contraction $T$ has a minimal isometric
  dilation $V$, which is uniquely
determined up to an isomorphism.
 Let $\Delta_i:\cH\to \cD \otimes F^2(H_n)$ be defined by
$$
\Delta_i h:=  [\Delta_{T}(
 \underbrace{0,\ldots, 0,}_{\text{$i-1$}\ times}
 h, 0,\dots,0)\otimes 1]\oplus 0\oplus 0\cdots.
$$
Consider the Hilbert space $\cK:=\cH\oplus (\cD \otimes F^2(H_n))$
and embed $\cH$ and  $\cD$ in $\cK$ in the natural way. For each
$i=1,\ldots, n$,  define the operator $V_i:\cK\to\cK$ by
\begin{equation}
\label{iso-dil}
 V_i(h\oplus (\xi\otimes d)):= T_ih \oplus [\Delta_i
h +( I_{\cD}\otimes S_i)(\xi\otimes d)]
\end{equation}
for any $h\in \cH$, \, $\xi\in F^2(H_n)$ , \,$d\in \cD$, where
$S_1,\ldots, S_n$ are the left creation operators on the full Fock
space $F^2(H_n)$. The $n$-tuple  $V:=(V_1,\ldots, V_n)$  is a
realization of the minimal isometric dilation of $T$. Note that
\begin{equation}\label{dilmatr}
V_i=\left[\begin{matrix} T_i& 0\\
\Delta_i&  I_\cD\otimes S_i
\end{matrix}\right]
\end{equation}
with respect to the decomposition $\cK=\cH\oplus [\cD \otimes
F^2(H_n)]$.

Let us introduce our   {\it generalized  noncommutative commutant
lifting} ({\bf GNCL}) problem.

A lifting {\it data set} $\{A,T,V,C,Q\}$  for the  {\bf GNCL}
problem is defined as follows.
  Let $T:=(T_1,\dots, T_n)$, $T_i\in B(\cH)$, be a row
contraction and let $V:=(V_1,\ldots, V_n)$, $V_i\in B(\cK)$, be the
minimal isometric dilation of  $T$ on a Hilbert space $\cK\supset
\cH$ given by \eqref{dilmatr}.
  Let $Q:=(Q_1,\ldots, Q_n)$,  $Q_i\in B(\cG_i,\cX)$,  and
 $C:=(C_1,\ldots, C_n)$, $C_i\in
B(\cG_i,\cX)$,  be such  that
\begin{equation}
\label{C-Q}
 \left[\delta_{ij} C_i^* C_j\right]_{n\times n}\leq \left[ Q_i^*
 Q_j\right]_{n\times n}.
 \end{equation}
Let $A\in B(\cX,\cH)$ be a contraction such that
\begin{equation}
\label{intert-rel} T_i AC_i=AQ_i, \qquad i=1,\ldots,n.
\end{equation}
We say that   $B$ is a {\it contractive interpolant} for $A$ with
respect to $\{A,T,V,C,Q\}$ if  $B\in B(\cX, \cK)$  is a contraction
satisfying the conditions
\begin{equation*}
%\label{v-i-r}
P_\cH B=A\quad \text{ and } \
 V_iBC_i=BQ_i,\quad i=1,\ldots,n,
\end{equation*}
where $P_\cH$ is the orthogonal projection from $\cK$ onto $\cH$.

\smallskip

{\bf The GNCL problem is  to find   contractive interpolats $B$  of
$A$ with respect to the data set $\{A,T,V,C,Q\}$.}

 \smallskip

  Note that $B$ satisfies  relation $P_\cH B=A$  if and only if
   it  has a matrix decomposition
  \begin{equation}\label{decomp}
  B=\left[\begin{matrix}A\\
  \Gamma D_{A}
  \end{matrix}
  \right]: \cX\to \cH\oplus [\cD\otimes F^2(H_n)],
  \end{equation}
 where $\Gamma:\cD_{A}\to\cD\otimes F^2(H_n)$ is a contraction.
 Here $D_A:=(I_\cX-A^*A)^{{1/2}}$ and $\cD_A:=\overline{D_A\cX}$.
 We mention that $B$ and $\Gamma$ determine each other uniquely.
Note that $B$ satisfies the equations  $V_iBC_i=BQ_i$ for $
i=1,\ldots,n$, if and only if
$$
\left[\begin{matrix} T_i& 0\\
\Delta_i&  I_\cD\otimes  S_i
\end{matrix}\right] \left[\begin{matrix}A\\
  \Gamma D_{A}
  \end{matrix}
  \right] C_i=\left[\begin{matrix}A\\
  \Gamma D_{A}
  \end{matrix}
  \right] Q_i,\quad i=1,\ldots, n,
  $$
  which, due to relation \eqref{intert-rel}, is equivalent to
  \begin{equation}
  \label{DSGA}
\Delta_iAC_i+( I_\cD\otimes S_i) \Gamma D_A C_i=\Gamma D_A Q_i,
\quad i=1,\ldots, n.
  \end{equation}
Therefore, the {\bf GNCL} problem  is equivalent to finding
contractions $\Gamma :\cD_A\to \cD\otimes F^2(H_n)$ such that
relation  \eqref{DSGA} holds. Using relations \eqref{C-Q} and
\eqref{intert-rel}, we deduce that
\begin{equation*}
\begin{split}
\left\|\sum_{i=1}^n D_A Q_i y_i\right\|^2 &= \left\|\sum_{i=1}^n Q_i
 y_i\right\|^2-\left\|\sum_{i=1}^n A Q_i y_i\right\|^2\\
 &\geq
 \sum_{i=1}^n \left\|C_i y_i\right\|^2-\left\|\sum_{i=1}^n T_iA Q_i y_i\right\|^2\\
&= \sum_{i=1}^n \left\|AC_i y_i\right\|^2 -\left\|\sum_{i=1}^n T_iA
Q_i y_i\right\|^2 +\sum_{i=1}^n \left\|C_i y_i\right\|^2-
\left\|\sum_{i=1}^n A C_i y_i\right\|^2\\
&= \left\| \Delta_T\left(\oplus_{i=1}^n AC_iy_i\right)\right\|^2+
\sum_{i=1}^n \left\|D_AC_iy_i\right\|^2
\end{split}
\end{equation*}
for any $y_i\in \cG_i$, $i=1,\ldots,n$. Consider the subspace
$\cF\subset \cD_A$  given by
\begin{equation}
\label{F}
 \cF:=\left\{ \sum_{i=1}^n D_A Q_i y_i:\ y_i\in \cG_i,  \
i=1,\ldots, n\right\}^{-}.
\end{equation}
Due to the estimations above, we can  introduce the operator
$\Omega:\cF\to \cD\oplus
\cD_A^{(n)}$    by  $\Omega:=\left[\begin{matrix} \Omega_1\\
\Omega_2\end{matrix} \right]$, where   $\Omega_1$ and  $\Omega_2$
are defined  as follows:

\begin{equation}
\label{OmOm}
\begin{split}
  \Omega_1&:\cF\to \cD,\qquad \qquad \quad \Omega_1\left(\sum_{i=1}^n D_A
Q_i y_i\right):=\Delta_T\left(\oplus_{i=1}^n AC_iy_i\right) \text{ and }\\
\Omega_2&:\cF\to\oplus_{i=1}^n \cD_A,\qquad
\Omega_2\left(\sum_{i=1}^n D_A Q_i y_i\right):=\oplus_{i=1}^n
D_AC_iy_i.
\end{split}
\end{equation}
Consequently,  $\Omega$ is a contraction. We remark  that $\Omega$
is an isometry if and only if we have equality in \eqref{C-Q}. Now,
notice that
$$
\sum_{i=1}^n \Delta_i(AC_iy_i)=\Delta_T(\oplus_{i=1}^n
AC_iy_i)\otimes 1,\qquad y_i\in \cG_i.
$$
It is clear that relation \eqref{DSGA} is equivalent to
\begin{equation}
\label{another-eq} \Delta_T(\oplus_{i=1}^n AC_iy_i)\otimes
1+\sum_{i=1}^n( I_\cD\otimes S_i) \Gamma D_A C_i
y_i=\sum_{i=1}^n\Gamma D_A Q_i y_i,\qquad y_i\in \cG_i,
\end{equation}
which is equivalent to
$$
\Omega_1\left(\sum_{i=1}^n D_A Q_i y_i\right)\otimes 1 + \left[
I_\cD\otimes S_1,\ldots, I_\cD\otimes S_n\right](\oplus_{i=1}^n
\Gamma)\Omega_2\left(\sum_{i=1}^n D_A Q_i y_i\right)
=\Gamma\left(\sum_{i=1}^n D_A Q_i y_i\right).
$$

Therefore, we have
\begin{equation}
\label{ooga} E_\cD\Omega_1+\left[ I_\cD\otimes S_1,\ldots,
I_\cD\otimes
S_n\right] \left[\begin{matrix} \Gamma P_1\\
\vdots\\
\Gamma P_n\end{matrix} \right] \Omega_2=\Gamma|_\cF,
\end{equation}
where $E_\cD:\cD\to \cD\otimes F^2(H_n)$  is defined by $E_\cD
y=y\otimes 1$ and $P_j: \cD_A^{(n)}\to \cD_A$ is the orthogonal
projection of $ \cD_A^{(n)}$ onto its  $j$-th coordinate.

Let $B$ be a fixed solution of the {\bf GNCL} problem  and let
$\Gamma:\cD_A\to \cD\otimes F^2(H_n)$ be the unique contraction
determined by  $B$  (see \eqref{decomp}). Since relation
\eqref{DSGA} holds and $S_1,\ldots, S_n$ are isometries with
orthogonal ranges, we deduce that
\begin{equation*}
\begin{split}
\left\|\sum_{i=1}^n D_\Gamma D_A Q_iy_i\right\|^2&=
\left\|\sum_{i=1}^n D_A Q_iy_i\right\|^2-\left\|\sum_{i=1}^n \Gamma
D_A Q_iy_i\right\|^2\\
&= \left\|\sum_{i=1}^n D_A Q_iy_i\right\|^2 -\left\|\sum_{i=1}^n
\Delta_i A C_iy_i\right\|^2- \left\|\sum_{i=1}^n (I_\cD\otimes
S_i)\Gamma D_A C_iy_i\right\|^2\\
&= \left\|\sum_{i=1}^n D_A Q_iy_i\right\|^2 - \left\|
\Delta_T\left(\oplus_{i=1}^n AC_iy_i\right)\right\|^2 -
\sum_{i=1}^n \left\| \Gamma D_A C_iy_i\right\|^2\\
&= \sum_{i=1}^n\left\| D_\Gamma D_A C_iy_i\right\|^2
+\left\|\sum_{i=1}^n D_A Q_iy_i\right\|^2- \left\|
\Delta_T\left(\oplus_{i=1}^n AC_iy_i\right)\right\|^2-
\sum_{i=1}^n \left\| D_A C_iy_i\right\|^2\\
&= \sum_{i=1}^n\left\| D_\Gamma D_A C_iy_i\right\|^2
+\left\|\sum_{i=1}^n D_A Q_iy_i\right\|^2-\left\|\Omega \left(
\sum_{i=1}^n D_A Q_i y_i\right)\right\|^2\\
&=\sum_{i=1}^n\left\| D_\Gamma D_A C_iy_i\right\|^2 +
\left\|D_\Omega D_A \left( \sum_{i=1}^n  Q_i y_i\right)\right\|^2\\
&\geq \sum_{i=1}^n\left\| D_\Gamma D_A C_iy_i\right\|^2
\end{split}
\end{equation*}
for any $y_i\in \cG_i$, $ i=1,\ldots,n$.  Therefore,
$$
\left\|\sum_{i=1}^n D_\Gamma D_A Q_iy_i\right\|\geq
\left(\sum_{i=1}^n \left\|D_\Gamma D_A
C_iy_i\right\|^2\right)^{1/2},\quad y_i\in \cG_i,
$$
where the equality  holds if and only if $\Omega$ is an isometry.
Consequently, we can define a contraction $\Lambda:
\cF_\Gamma:=\overline{\cD_\Gamma\cF}\to  \cD_\Gamma^{(n)}$ by
setting
\begin{equation}
\label{Lambd} \Lambda\left(\sum_{i=1}^n D_\Gamma D_A
Q_iy_i\right):=\oplus_{i=1}^n (D_\Gamma D_A C_iy_i),\qquad y_i\in
\cG_i.
\end{equation}
Using the definition of $\Omega_2$, we deduce that
\begin{equation}
\label{Lambd2} \Lambda D_\Gamma x=\left(\oplus_{i=1}^n
D_\Gamma\right) \Omega_2x, \qquad x\in \cF.
\end{equation}
We remark that $\Lambda$ is an isometry if and only if $\Omega$ is
an isometry.

We introduce the  noncommutative Schur class $\cS_{\bf
ball}^{\Lambda}(B(\cD_\Gamma, \cD_\Gamma^{(n)}))$  of all bounded
free holomorphic functions $\Phi\in H_{\bf
ball}^\infty(B(\cD_\Gamma, \cD_\Gamma^{(n)}))$ with
$\|\Phi\|_\infty\leq 1$ such that $C|_{\cF_\Gamma}=\Lambda$. More
precisely, if $C$ has the representation $\Phi(X_1,\ldots,
X_n)=\sum_{\alpha\in \FF_n^+} C_{(\alpha)}\otimes X_\alpha$ for some
coefficients
 $C_{(\alpha)}\in B(\cD_\Gamma, \cD_\Gamma^{(n)})$, the
 latter condition means  $C_{(0)}|_{\cF_\Gamma}=\Lambda$ and
 $C_{\alpha)}|_{\cF_\Gamma}=0$  if $|\alpha|\geq 1$.
Equivalently,
 $\Phi(rR_1,\ldots, rR_n)|_{\cF_\Gamma\otimes 1}=
 (\Lambda\otimes I_{F^2(H_n)})|_{\cF_\Gamma\otimes 1}
 $,
 i.e.,
 \begin{equation}
 \label{cond-Phi}
 \Phi(rR_1,\ldots, rR_n)(x\otimes 1)=\Lambda x\otimes 1
 \end{equation}
  for any $x\in \cF_\Gamma$ and $r\in [0,1)$. Moreover, notice  also that the
  latter condition is equivalent to $\widetilde
  \Phi|_{\cF_\Gamma\otimes F^2(H_n)}=\Lambda \otimes I_{F^2(H_n)}$,
  where $\widetilde \Phi$ is the boundary function of $\Phi$.

We remark  that  the set $\cS_{\bf ball}^{\Lambda}(\cD_\Gamma,
\cD_\Gamma^{(n)})$ is nonempty. Indeed, we can take $C=\Lambda
P_{\cF_\Gamma}\otimes I$, where $P_{\cF_\Gamma}$ is the orthogonal
projection of $\cD_\Gamma$ onto $\cF_\Gamma$.

We say that a bounded free holomorphic function $\Psi\in H_{\bf
ball}^\infty(B(\cD_A, \cD\oplus  \cD_A^{(n)}))$ is a {\it Schur
function associated with the data set} $\{A,T,V,C,Q\}$  if
$\|\Psi\|_\infty\leq 1$ such that $\Psi|_{\cF}=\Omega$.
Equivalently,
 \begin{equation}
\label{Schur-fun}
 \|\Psi(rR_1,\ldots, rR_n)\|\leq 1\quad  \text{  and  }\quad
 \Psi(rR_1,\ldots, rR_n)(y\otimes 1)=\Omega y \otimes 1
 \end{equation}
  for any $r\in [0,1)$ and $y\in \cF$.
We denote by $\cS_{\bf ball}^{\Omega}(B(\cD_A, \cD\oplus
\cD_A^{(n)}))$ the set of all Schur functions associated with the
data set $\{A,T,V,C,Q\}$.

Let $B$ be a solution of the {\bf GNCL} problem with the data set
$\{A,T,V,C,Q\}$ and consider the contraction $\Gamma:\cD_A\to
\cD\otimes F^2(H_n)$ uniquely  determined by $B$ (see
\eqref{decomp}). Let $\Theta \in H^2_{\bf ball} (B(\cD_A, \cD))$ be
the free holomorphic function   with  symbol $\Gamma$. Define the
map
$$
J_\Gamma: \cS_{\text{\bf ball}}(B(\cD_{\Gamma}, \cD_{\Gamma}^{(n)}
))\to \cS_{\text{\bf ball}}(B(\cD_A,\cD\oplus \cD_A^{(n)}))
$$
  by setting
\begin{equation}
\label{J-ga}
 J_\Gamma \left[\begin{matrix}
\varphi_1\\
\vdots\\
\varphi_n\end{matrix}\right]:=\left[\begin{matrix} L\\
M_1\\
\vdots\\
M_n\end{matrix}\right],
\end{equation}
where $L\in H_{\text{\bf ball}} (B(\cD_A,\cD))$ is  given  by
\begin{equation}
\label{L-def2}
 L(Z):=2\Theta(Z)[F(Z)+I]^{-1}, \quad Z\in [B(\cZ)^n]_1,
\end{equation}
 the free holomorphic function $F\in H_{\text{\bf ball}} (B(\cD_A ))$ is  defined  by
 \begin{equation}\label{F-def2}
 \begin{split}
 F(Z)&:= (\Gamma^*\otimes I_\cZ)
\left(I_{\cD\otimes F^2(H_n)\otimes \cZ}+ \sum_{i=1}^n I_\cD\otimes
S_i^*\otimes Z_i\right) \\
&\qquad \times \left(I_{\cD\otimes F^2(H_n)\otimes \cZ}-
\sum_{i=1}^n I_\cD\otimes
S_i^*\otimes Z_i\right)^{-1}(\Gamma\otimes I_\cZ)\\
 &\qquad +(D_{\Gamma}\otimes I_\cZ)
 \left[ I+\sum_{i=1}^n (I_{\cD_A}\otimes Z_i)\varphi_i(Z)\right]
 \left[ I-\sum_{i=1}^n (I_{\cD_A}\otimes Z_i)\varphi_i(Z)\right]^{-1}(D_{\Gamma}\otimes
 I_\cZ),
 \end{split}
 \end{equation}
and $M_1,\ldots, M_n\in H_{\text{\bf ball}} (B(\cD_A ))$ are
uniquely determined by the equation
\begin{equation}
\label{M-def2}(I_{\cD_A}\otimes Z_1)M_1(Z)+\cdots +(I_{\cD_A}\otimes
Z_n)M_n(Z)= [F(Z)-I][F(Z)+ I]^{-1}
\end{equation}
for any   $Z:=(Z_1,\ldots, Z_n)\in [B(\cZ)^n]_1$. Here $\cZ$ is an
infinite dimensional Hilbert space.

In what follows we need  the following   lemma, which  is the core
of the main result of this section.

\begin{lemma}
\label{main-ingredient} Let  $\{A,T,V,C,Q\}$ be a  data set for the
{\bf GNCL} problem. Let $\Gamma:\cD_{A}\to \cD\otimes F^2(H_n)$ be a
contraction and let $\Phi:=\left[\begin{matrix}
\varphi_1\\
\vdots\\
\varphi_n\end{matrix}\right]$ be in $\cS_{\text{\bf
ball}}(B(\cD_{\Gamma}, \cD_{\Gamma}^{(n)} ))$. Define
$
\Psi:=\left[\begin{matrix} L\\
M_1\\
\vdots\\
M_n\end{matrix}\right]$ by relations \eqref{L-def2}, \eqref{F-def2},
and  \eqref{M-def2}. Then the following statements hold:
\begin{enumerate}
\item[(i)] If  $\Gamma$ satisfies  relation \eqref{DSGA}, then the
free holomorphic function $M:=\left[\begin{matrix}
M_1\\
\vdots\\
M_n\end{matrix}\right]$ has the property that $M|_\cF=\Omega_2$ if
and only if $\Phi\in\cS^\Lambda_{\text{\bf ball}}(B(\cD_{\Gamma},
\cD_{\Gamma}^{(n)} ))$;

\item[(ii)]
$\Psi$  is in $\cS_{\bf ball}^{\Omega}(B(\cD_A, \cD\oplus
\cD_A^{(n)}))$ if and only if \ $\Gamma$ satisfies \eqref{DSGA} and
$\Phi$ is in $\cS_{\bf ball}^{\Lambda}(B(\cD_\Gamma,
\cD_\Gamma^{(n)}))$.
\end{enumerate}
\end{lemma}

\begin{proof}
Let $\Theta \in H^2_{\bf ball} (B(\cD_A, \cD))$
 be  the free holomorphic function
  with symbol
$\Gamma$. Due to relation \eqref{F-def2}, we have
\begin{equation*}
\begin{split}
F(Z) &=\Gamma^* \Gamma\otimes I_\cZ + 2(\Gamma^*\otimes I_\cZ)
\left(I_{\cD\otimes F^2(H_n)\otimes \cZ}- \sum_{i=1}^n I_\cD\otimes
S_i^*\otimes Z_i\right)^{-1}\left(\sum_{i=1}^n I_\cD\otimes
S_i^*\otimes Z_i\right)
(\Gamma\otimes I_\cZ)\\
&+ D_\Gamma^2\otimes I_\cZ+2(D_\Gamma\otimes I_\cZ) \left[
I-\sum_{i=1}^n (I_{\cD_A}\otimes Z_i)\varphi_i(Z)\right]^{-1}\left[
\sum_{i=1}^n (I_{\cD_A}\otimes
Z_i)\varphi_i(Z)\right](D_{\Gamma}\otimes
 I_\cZ)
\end{split}
\end{equation*}
for any $Z:=(Z_1,\ldots, Z_n)\in [B(\cZ)^n]_1$.
  Since
$\Gamma^* \Gamma +D_\Gamma^2=I$, we deduce that
\begin{equation}\label{F-I}
\begin{split}
F(Z)-I&=2(\Gamma^*\otimes I_\cZ) \left(I_{\cD\otimes F^2(H_n)\otimes
\cZ}- \sum_{i=1}^n I_\cD\otimes S_i^*\otimes
Z_i\right)^{-1}\left(\sum_{i=1}^n I_\cD\otimes
S_i^*\otimes Z_i\right)(\Gamma\otimes I_\cH)\\
&+  2(D_\Gamma\otimes I_\cH) \left[ I-\sum_{i=1}^n (I_{\cD_A}\otimes
Z_i)\varphi_i(Z)\right]^{-1}\left[ \sum_{i=1}^n (I_{\cD_A}\otimes
Z_i)\varphi_i(Z)\right](D_{\Gamma}\otimes
 I_\cZ)
\end{split}
\end{equation}
for any $Z:=(Z_1,\ldots, Z_n)\in [B(\cZ)^n]_1$.
 Similarly, we obtain
\begin{equation}\label{F+I}
\begin{split}
F(Z)+I&=2(\Gamma^*\otimes I_\cZ)  \left(I_{\cD\otimes
F^2(H_n)\otimes \cZ}- \sum_{i=1}^n I_\cD\otimes
S_i^*\otimes Z_i\right)^{-1}(\Gamma\otimes I_\cZ)\\
&+  2(D_\Gamma\otimes I_\cZ) \left[ I-\sum_{i=1}^n (I_{\cD_A}\otimes
Z_i)\varphi_i(Z)\right]^{-1} (D_{\Gamma}\otimes
 I_\cZ)
\end{split}
\end{equation}
for any $Z:=(Z_1,\ldots, Z_n)\in [B(\cZ)^n]_1$.

Now, assume that $\Gamma$ satisfies relation \eqref{ooga} (which is
equivalent to \eqref{DSGA}). Let us show that
$M:=\left[\begin{matrix}
M_1\\
\vdots\\
M_n\end{matrix}\right]$ has the property that $M|_\cF=\Omega_2$ if
and only if $\Phi\in\cS^\Lambda_{\text{\bf ball}}(B(\cD_{\Gamma},
\cD_{\Gamma}^{(n)} ))$.
 Let  $x\in \cF$ and $r\in [0,1)$. Since $S_1,\ldots, S_n$ are  isometries with orthogonal
 ranges,  relation \eqref{ooga}
implies
\begin{equation*}
\begin{split}
&\left(\sum_{i=1}^n I_\cD\otimes S_i^*\otimes rR_i\right)
(\Gamma\otimes I_{F^2(H_n)})(x\otimes 1)\\
&\qquad= \sum_{i=1}^n\left[(I_\cD\otimes S_i^*)
\left(E_\cD\Omega_1x+\left[ I_\cD\otimes S_1,\ldots, I_\cD\otimes
S_n\right] \left[\begin{matrix} \Gamma P_1\\
\vdots\\
\Gamma P_n\end{matrix} \right] \Omega_2 x\right)\otimes
re_i\right]\\
&\qquad= \sum_{i=1}^n(\Gamma P_i\Omega_2x\otimes
re_i)=(\Gamma\otimes I_{F^2(H_n)})\sum_{i=1}^n(  P_i\Omega_2x\otimes
re_i).
\end{split}
\end{equation*}
Using  relation \eqref{F-I}, we have \ $F(rR_1,\ldots,
rR_n)-I=A_r+B_r$, $r\in [0,1)$,  where $rR:=(rR_1,\ldots, rR_n)$,
\begin{equation*}
\begin{split}
A_r&:=2(\Gamma^*\otimes I_{F^2(H_n)}) \left(I_{\cD\otimes
F^2(H_n)\otimes \cH}- \sum_{i=1}^n I_\cD\otimes S_i^*\otimes
rR_i\right)^{-1}\left(\sum_{i=1}^n I_\cD\otimes
S_i^*\otimes rR_i\right)(\Gamma\otimes I_{F^2(H_n)}), \text{and} \\
B_r&:=  2(D_\Gamma\otimes I_{F^2(H_n)}) \left[ I-\sum_{i=1}^n
(I_{\cD_A}\otimes rR_i)\varphi_i(rR)\right]^{-1}\left[ \sum_{i=1}^n
(I_{\cD_A}\otimes rR_i)\varphi_i(rR)\right](D_{\Gamma}\otimes
 I_{F^2(H_n)}).
\end{split}
\end{equation*}

  Taking into account the above calculations, we
deduce that
\begin{equation*}
\begin{split}
&[F(rR_1,\ldots, rR_n)-I](x\otimes 1)\\
&=2(\Gamma^*\otimes I_{F^2(H_n)}) \left(I_{\cD\otimes
F^2(H_n)\otimes \cH}- \sum_{i=1}^n I_\cD\otimes S_i^*\otimes
rR_i\right)^{-1} (\Gamma\otimes
I_{F^2(H_n)})\sum_{i=1}^n(  P_i\Omega_2x\otimes re_i)\\
&\quad +B_r(x\otimes 1)
\end{split}
\end{equation*}
for $x\in \cF$. Hence and  due to  \eqref{F+I}, we obtain
\begin{equation*}
\begin{split}
&[F(rR_1,\ldots, rR_n)-I](x\otimes 1)\\
& =[F(rR_1,\ldots,
rR_n)+I]\sum_{i=1}^n(  P_i\Omega_2x\otimes re_i)\\
&\quad -
 2(D_\Gamma\otimes I_{F^2(H_n)}) \left[ I-\sum_{i=1}^n
(I_{\cD_A}\otimes rR_i)\varphi_i(rR)\right]^{-1} (D_{\Gamma}\otimes
 I_{F^2(H_n)}) \sum_{i=1}^n(  P_i\Omega_2x\otimes re_i)
  +B_r(x\otimes 1)\\
&=[F(rR_1,\ldots, rR_n)+I]\sum_{i=1}^n(  P_i\Omega_2x\otimes re_i)
 + 2(D_\Gamma\otimes I_{F^2(H_n)}) \left[ I-\sum_{i=1}^n
(I_{\cD_A}\otimes rR_i)\varphi_i(rR)\right]^{-1}\chi_r,
\end{split}
\end{equation*}
where
$$\chi_r:= \sum_{i=1}^n (I_{\cD_A}\otimes
rR_i)\varphi_i(rR)(D_{\Gamma}\otimes
 I_{F^2(H_n)})(x\otimes 1)-(D_{\Gamma}\otimes
 I_{F^2(H_n)})\sum_{i=1}^n(  P_i\Omega_2x\otimes re_i).
 $$
 Consequently, we have
\begin{equation}
\begin{split}
\label{FF}
 &[F(rR_1,\ldots, rR_n)+I]^{-1}[F(rR_1,\ldots,
rR_n)-I](x\otimes1) \\
&\qquad = \sum_{i=1}^n(  P_i\Omega_2x\otimes re_i)
 +2[F(rR_1,\ldots, rR_n)+I]^{-1}(D_\Gamma\otimes
I_{F^2(H_n)}) \\
&\qquad \qquad \qquad \times \left[ I-\sum_{i=1}^n (I_{\cD_A}\otimes
rR_i)\varphi_i(rR)\right]^{-1}\chi_r.
\end{split}
\end{equation}
If  $\Phi\in\cS^\Lambda_{\text{\bf ball}}(B(\cD_{\Gamma},
\cD_{\Gamma}^{(n)} ))$, then due to relation \eqref{cond-Phi}, we
have
$$
\Phi(rR_1,\ldots, rR_n)(y\otimes 1)=\Lambda y\otimes 1
  $$
  for any $y\in \cF_\Gamma$ and $r\in [0,1)$.
  Using the definition of $\cF_\Gamma$ and relations \eqref{Lambd},
  \eqref{Lambd2}, we deduce that, for any $x\in \cF$,
\begin{equation*}
\begin{split}
\varphi_j(rR_1,\ldots, rR_n)(D_\Gamma x\otimes 1)&= (P_j\otimes
I_{F^2(H_n)}) \Phi(rR_1,\ldots, rR_n)(D_\Gamma x\otimes 1)\\
&= (P_j\otimes I_{F^2(H_n)})  (\Lambda D_\Gamma x\otimes 1) \\
&=P_j \left(\oplus_{i=1}^n D_\Gamma\right) \Omega_2x\otimes 1\\
&=D_\Gamma P_j\Omega_2x\otimes 1
\end{split}
\end{equation*}
for any $x\in \cF$. Now, it is clear that $\chi_r=0$. Due to
relation \eqref{FF}, we have
$$[F(rR_1,\ldots, rR_n)+I]^{-1}[F(rR_1,\ldots, rR_n)-I](x\otimes1)\\
= \sum_{i=1}^n(  P_i\Omega_2x\otimes re_i).
$$
Hence and using \eqref{M-def2}, we have
\begin{equation*}
\begin{split}
M_j(rR_1,\ldots, rR_n)(x\otimes 1)&= \frac{1}{r} (I_{\cD_A}\otimes
R_j^*)[F(rR_1,\ldots, rR_n)+I]^{-1}[F(rR_1,\ldots,
rR_n)-I](x\otimes1)\\
&= P_j\Omega_2x\otimes 1
\end{split}
\end{equation*}
for any $x\in \cF$ and $j=1,\ldots, n$. Consequently,
$M(rR_1,\ldots, rR_n)(x\otimes 1)=\Omega_2x\otimes 1$  for any $x\in
\cF$, i.e., $M|_\cF=\Omega_2$.

Conversely, if $M|_\cF=\Omega_2$, then, due to relation
\eqref{M-def2}, we have
$$
[F(rR_1,\ldots, rR_n)+I]^{-1}[F(rR_1,\ldots, rR_n)-I](x\otimes1)=
\sum_{i=1}^n P_j\Omega_2x\otimes r e_i.
$$
Using relation \eqref{FF}, we obtain
$$
2[F(rR_1,\ldots, rR_n)+I]^{-1}(D_\Gamma\otimes I_{F^2(H_n)}) \left[
I-\sum_{i=1}^n (I_{\cD_A}\otimes
rR_i)\varphi_i(rR)\right]^{-1}\chi_r=0.
$$
Since $\chi_r$ has the range in $\cD_\Gamma\otimes F^2(H_n)$, the
operator $I-\sum_{i=1}^n (I_{\cD_A}\otimes rR_i)\varphi_i(rR)$  in
invertible on the Hilbert space $\cD_\Gamma\otimes F^2(H_n)$, and
$D_\Gamma\otimes I_{F^2(H_n)}$ is one-to-one on $\cD_\Gamma\otimes
F^2(H_n)$, we deduce that $\chi_r=0$.  Consequently, we have
$$
\sum_{i=1}^n (I_{\cD_\Gamma}\otimes rR_i)\left[
\varphi_i(rR_1,\ldots, rR_n)(D_\Gamma x\otimes 1)-D_\Gamma
P_i\Omega_2 x\otimes 1\right]=0, \quad x\in\cF.
$$
Since $R_1,\ldots, R_n$ are isometries with orthogonal ranges, we
deduce that
$$
\varphi_i(rR_1,\ldots, rR_n)(D_\Gamma x\otimes 1)=D_\Gamma
P_i\Omega_2x\otimes 1,\qquad i=1,\ldots,n,
$$
for any $x\in \cF$. On the other hand, due to \eqref{Lambd} and
  \eqref{Lambd2}, we have
$$
D_\Gamma P_j\Omega_2x\otimes 1= (P_j\otimes I_{F^2(H_n)})  (\Lambda
D_\Gamma x\otimes 1),\qquad i=1,\ldots,n.
$$
Combining these relations, we deduce that $ \Phi(rR_1,\ldots,
rR_n)(y\otimes 1)=\Lambda y\otimes 1
  $
  for any $y\in \cF_\Gamma$ and $r\in [0,1)$. Therefore,
  $\Phi\in\cS^\Lambda_{\text{\bf ball}}(B(\cD_{\Gamma},
\cD_{\Gamma}^{(n)} ))$, which proves part (i).

Assume now  that $\Gamma$ satisfies \eqref{DSGA}   and
$\Phi:=\left[\begin{matrix}
\varphi_1\\
\vdots\\
\varphi_n\end{matrix}\right]$ is in $\cS^\Lambda_{\text{\bf
ball}}(B(\cD_{\Gamma}, \cD_{\Gamma}^{(n)} ))$. The result of   part
(i)   shows that $M:=\left[\begin{matrix}
M_1\\
\vdots\\
M_n\end{matrix}\right]$ has the property that $M|_\cF=\Omega_2$. In
what follows we will use  the fact that
$J_\Gamma\left[\begin{matrix}
\varphi_1\\
\vdots\\
\varphi_n\end{matrix}\right]=\left[\begin{matrix} L\\
M_1\\
\vdots\\
M_n\end{matrix}\right]$ and that  relations  \eqref{L-def2},
\eqref{F-def2},   and \eqref{M-def2}  hold. First, notice that
\eqref{M-def2} implies
\begin{equation*}
\begin{split}
I -\sum_{i=1}^n (I_{\cD_A}\otimes
rR_i)M_i (rR_1,\ldots, rR_n)&=I-[F(rR_1,\ldots, rR_n)-I][F(rR_1,\ldots, rR_n)+I]^{-1}\\
&=2[F(rR_1,\ldots, rR_n)+I]^{-1}
\end{split}
\end{equation*}
for any $r\in [0,1)$. Hence and using relation \eqref{L-def2}, we
get
\begin{equation}
\label{L-Phi}
\begin{split}
L(rR_1,\ldots, rR_n)&=2\Theta(rR_1,\ldots, rR_n)[F(rR_1,\ldots, rR_n)+I]^{-1}\\
&=\Theta(rR_1,\ldots, rR_n)\left[I -\sum_{i=1}^n (I_{\cD_A}\otimes
rR_i)M_i(rR_1,\ldots, rR_n)\right].
\end{split}
\end{equation}
Therefore,   since $M|_\cF=\Omega_2$,  we have
\begin{equation*}
\begin{split}
&L(rR_1,\ldots, rR_n)(x\otimes 1)\\
&=\Theta(rR_1,\ldots,
rR_n)(x\otimes 1) -\Theta(rR_1,\ldots, rR_n)\sum_{i=1}^n
(I_{\cD_A}\otimes rR_i)M_i(rR_1,\ldots, rR_n)(x\otimes 1)\\
&=\Theta(rR_1,\ldots, rR_n)(x\otimes 1) -\Theta(rR_1,\ldots,
rR_n)\sum_{i=1}^n (P_i\Omega_2 x\otimes re_i)
\end{split}
\end{equation*}
for any $x\in \cF$.  Since $L$ is a bounded free holomorphic
function, $\tilde L:=\text{\rm SOT-}\lim_{r\to 1} L(rR)$ exists and
it is in the operator space  $B(\cD_A,\cD)\bar\otimes  R_n^\infty$.
Taking $r\to 1$ in the relation above and using \eqref{ooga}, we
obtain
\begin{equation*}
\begin{split}
L(rR_1,\ldots, rR_n)(x\otimes 1)&= \Gamma x-\lim_{r\to
1}\Theta(rR_1,\ldots, rR_n)\sum_{i=1}^n
(P_i\Omega_2 x\otimes re_i)\\
&=\Omega_1 x\otimes 1+ \sum_{i=1}^n (I_\cD\otimes S_i) \Gamma
P_i\Omega_2 x-\lim_{r\to 1}\Theta(rR_1,\ldots, rR_n)\sum_{i=1}^n
(P_i\Omega_2 x\otimes re_i).
\end{split}
\end{equation*}
Now, assume that $\Theta$ has the representation $\Theta(Z_1,\ldots,
Z_n)=\sum_{k=0}^\infty \sum_{|\alpha|=k} A_{(\alpha)} \otimes
Z_\alpha$  on $ [B(\cZ)^n]_1, $ with $\sum_{\alpha\in
F_n^+}A_{(\alpha)}^*A_{(\alpha)}\leq I$. Then $\Gamma
y=\sum_{\alpha\in F_n^+}A_{(\alpha)}y\otimes e_{\tilde \alpha}$,\
$y\in \cD_A$,  and
$$
\sum_{i=1}^n (I_\cD\otimes S_i) \Gamma P_i\Omega_2 x=
\sum_{\alpha\in F_n^+}A_{(\alpha)}P_i\Omega_2 x\otimes e_{g_i\tilde
\alpha}, \quad x\in \cF.
$$
On the other hand, we have
\begin{equation*}
\Theta(rR_1,\ldots, rR_n)\left(\sum_{i=1}^n (P_i\Omega_2 x\otimes
re_i)\right) =\sum_{i=1}^n \sum_{\alpha\in \FF_n^+} A_{(\alpha)} P_i
\Omega_2 x\otimes r^{|\alpha|+1} e_{g_i\tilde \alpha}.
\end{equation*}
Consequently,
\begin{equation}
\label{lim-The} \lim_{r\to 1}\Theta(rR_1,\ldots,
rR_n)\left(\sum_{i=1}^n (P_i\Omega_2 x\otimes re_i)\right)
=\sum_{i=1}^n (I_\cD\otimes S_i) \Gamma P_i\Omega_2 x.
\end{equation}
Hence, $\tilde L(x\otimes 1)= \Omega_1x\otimes 1$ for $x\in \cF$,
which implies $L|_\cF=\Omega_1$.

Conversely, assume that $
\Psi:=\left[\begin{matrix} L\\
M_1\\
\vdots\\
M_n\end{matrix}\right]$   is in $S_{\bf ball}^{\Omega}(B(\cD_A,
\cD\oplus   \cD_A^{(n)}))$ and let $M:=\left[\begin{matrix}
M_1\\
\vdots\\
M_n\end{matrix}\right]$. Then we have
\begin{equation*}
\begin{split}
L(rR_1,\ldots, rR_n)(y\otimes 1)&=\Omega_1y\otimes 1, \ y\in \cF, \
\text{ and }\\
M(rR_1,\ldots, rR_n)(y\otimes 1)&=\Omega_2y\otimes 1, \ y\in \cF.
\end{split}
\end{equation*}
Due to relation \eqref{L-Phi}, we deduce that
\begin{equation*}
\begin{split}
\Omega_1 y\otimes 1&= L(rR_1,\ldots, rR_n)(y\otimes
1)\\
&=\Theta(rR_1,\ldots, rR_n)(y\otimes 1) -\Theta(rR_1,\ldots,
rR_n)\sum_{i=1}^n (I_{\cD_A}\otimes rR_i)M_i(rR_1,\ldots
rR_n)(y\otimes 1)\\
&=\Theta(rR_1,\ldots, rR_n)(y\otimes 1) -\Theta(rR_1,\ldots,
rR_n)\sum_{i=1}^n (I_{\cD_A}\otimes rR_i) (P_i\Omega_2 y\otimes 1)
\end{split}
\end{equation*}
for any $y\in \cF$.  As before (see \eqref{lim-The}), taking $r\to
1$, we get
$$
\Omega_1y\otimes 1=\Gamma y-\sum_{i=1}^n (I_{\cD_A}\otimes rR_i)
(P_i\Omega_2 y\otimes 1), y\in \cF,
$$
which shows that $\Gamma$ satisfies relation \eqref{ooga}. Hence,
and using   part  (i), we deduce that $\Phi \in\cS_{\bf
ball}^{\Lambda}(B(\cD_\Gamma,  \cD_\Gamma^{(n)}))$. The proof is
complete.

\end{proof}

Now we can  prove  the following generalized  noncommutative
commutant lifting theorem, which is the   main result  of this
section.

\begin{theorem}
\label{main} Let  $\{A,T,V,C,Q\}$ be a  data set. Then any solution
of the   {\bf GNCL} problem is given by
 \begin{equation}
 \label{B-rep}
  B=\left[\begin{matrix}A\\
  \Gamma D_{A}
  \end{matrix}
  \right]: \cX\to \cH\oplus [\cD\otimes F^2(H_n)],
  \end{equation}
 where $\Gamma:\cD_{A}\to \cD\otimes  F^2(H_n)$ is the symbol of  a free holomorphic function
 $\Theta \in H^2_{\bf ball} (B(\cD_A, \cD))$
 given by
\begin{equation}
\label{Th-equ}
  \Theta(Z)=L(X)\left[ I_{\cD_A\otimes \cZ}-
\sum_{i=1}^n(I_{\cD_A}\otimes Z_i)M_i(Z)\right]^{-1}\quad  \text{
for
  } \ Z:=(Z_1,\ldots, Z_n)\in [B(\cZ)^n]_1,
\end{equation}
where $\left[\begin{matrix} L\\
M_1\\
\vdots\\
M_n\end{matrix}\right]$  is an arbitrary element in  the
noncommutative Schur class $\cS_{\bf ball}^{\Omega}(B(\cD_A,
\cD\oplus   \cD_A^{(n)}))$.
\end{theorem}
\begin{proof}
Assume that $\Psi:= \left[\begin{matrix} L\\
M_1\\
\vdots\\
M_n\end{matrix}\right]$  is an arbitrary element in  $\cS_{\bf
ball}^{\Omega}(B(\cD_A, \cD\oplus  \cD_A^{(n)}))$ and let $\Theta$
be given by \eqref{Th-equ}. According to Corollary \ref{caract-2},
$\Theta$ is a free holomorphic function in $H_{\text{\bf ball}}^2
(B(\cD_A,\cD))$ and $\|\Theta\|_2\leq 1$.
Using Theorem \ref{one-to-one}, we deduce that  $\Psi=J_\Theta\Phi$
for a unique $\Phi$ in $\cS_{\text{\bf ball}}(B(\cD_{\Gamma},
\cD_{\Gamma}^{(n)} ))$. Now, since $\Psi  \in S_{\bf
ball}^{\Omega}(B(\cD_A, \cD\oplus   \cD_A^{(n)}))$, we can use Lemma
\ref{main-ingredient} to deduce that $\Gamma$ satisfies relation
\eqref{DSGA}. Therefore, $B$ is a solution of the {\bf GNCL}
problem.

Conversely, assume that $B$ is a solution of the {\bf GNCL} problem.
Then $B$ has a representation \eqref{B-rep}, where $\Gamma:\cD_A\to
\cD\otimes F^2(H_n)$ is a contraction satisfying \eqref{DSGA}. We
recall that $\cS^\Lambda_{\text{\bf ball}}(B(\cD_{\Gamma},
\cD_{\Gamma}^{(n)} ))$ is nonempty. Let $\Phi\in
\cS^\Lambda_{\text{\bf ball}}(B(\cD_{\Gamma}, \cD_{\Gamma}^{(n)} ))$
and set
$\Psi:= \left[\begin{matrix} L\\
M_1\\
\vdots\\
M_n\end{matrix}\right]:=J_\Gamma \Phi$ (see \eqref{J-ga}). Applying
again Lemma \ref{main-ingredient}, we deduce that $\Psi$ is in
$\cS_{\bf ball}^{\Omega}(B(\cD_A, \cD\oplus   \cD_A^{(n)}))$. Now,
using Theorem \ref{one-to-one}, we deduce that $\Gamma$ is the
symbol of    a free holomorphic function $\Theta\in H_{\text{\bf
ball}}^2 (B(\cD_A,\cD))$  satisfying \eqref{Th-equ}. This completes
the proof.
\end{proof}

 To obtain a refinement of Theorem \ref{main}, we need the
 following result.

\begin{lemma}
\label{factorization} Let $\cM,\cL$ be Hilbert spaces and $\Phi$ be
a bounded free holomorphic function  with coefficients in
$B(\cM,\cL)$. Let $\cN$ be a subspace of $\cM$ and  $A\in
B(\cN,\cL)$. Then $\|\Phi\|_\infty\leq 1$ and $\Phi|_\cN=A$ if and
only if
\begin{equation}
\label{fact} \Phi=(AP_\cN\otimes I)+ (D_{A^*}\otimes I)\Psi
(P_{\cN^\perp}\otimes I)
\end{equation}
 for some $\Psi\in H^\infty_{\bf ball} (B(\cN^\perp, \cD_{A^*}))$
 with $\|\Psi\|_\infty\leq 1$, where $\cN^\perp:=\cM\ominus \cN$. Moreover, $\Phi$ and $\Psi$ in
 \eqref{fact} determine each other uniquely.
\end{lemma}
\begin{proof}
Let $\widetilde \Phi=\sum_{\alpha\in \FF_n^+} C_{(\alpha)}\otimes
R_\alpha$ be the Fourier  representation of $\Phi$. The condition
$\Phi|_\cN=A$  is equivalent to $C_{(0)}|_\cN=A$ and
$C_{(\alpha)}|_\cN=0$ for $|\alpha|\geq 1$.  The latter condition is
also equivalent to $ \widetilde \Phi|_{\cN\otimes F^2(H_n)}=A\otimes
I_{F^2(H_n)}. $ With respect to the decomposition $\cM\otimes
F^2(H_n)=[\cN\otimes F^2(H_n)]\oplus [\cN^\perp\otimes F^2(H_n)]$,
the operator $\widetilde \Phi:\cM\otimes F^2(H_n)\to \cL\otimes
F^2(H_n)$ has the matrix representation $\widetilde \Phi=[\widetilde
\Phi|_{\cN\otimes F^2(H_n)}\  \widetilde \Phi|_{\cN^\perp\otimes
F^2(H_n)}]$.
 Taking into account the structure of row contractions
(see \cite{GGK}), $[ A\otimes I_{F^2(H_n)}\ \widetilde \Phi
|_{\cN^\perp\otimes F^2(H_n)}]$ is a contraction if and only if
\begin{equation}
\label{formu}\widetilde  \Phi |_{\cN^\perp\otimes
F^2(H_n)}=(D_{A^*}\otimes I_{F^2(H_n)})\widetilde \Psi
\end{equation}
for a unique  contraction $\widetilde \Psi:\cN^\perp\otimes
F^2(H_n)\to \cD_{A^*}\otimes {F^2(H_n)}$. Moreover, $\widetilde
\Psi$ is unique. Since $\widetilde \Phi$ is a multi-analytic
operator, i.e.,  $ \widetilde \Phi(I_\cM \otimes S_i)=(I_\cL \otimes
S_i) \widetilde \Phi$,\ $ i=1,\ldots,n, $ so is its restriction
$\widetilde \Phi|_{\cN^\perp\otimes F^2(H_n)}$, i.e.,
$$
 \widetilde \Phi|_{\cN^\perp\otimes
F^2(H_n)}(I_{\cN^\perp} \otimes S_i)=(I_\cL \otimes S_i) \widetilde
\Phi|_{\cN^\perp\otimes F^2(H_n)},\qquad i=1,\ldots,n,
$$
Hence, we deduce that
$$
(D_{A^*}\otimes I_{F^2(H_n)})[\widetilde \Psi(I_{\cN^\perp} \otimes
S_i)-(I_{\cD_{A^*}}\otimes S_i)\widetilde \Psi]=0,\qquad i=1,\ldots,
n.
$$
Since $D_{A^*}\otimes I_{F^2(H_n)}$ is one-to-one on
$\cD_{A^*}\otimes F^2(H_n)$, we obtain $ \widetilde
\Psi(I_{\cN^\perp} \otimes S_i)=(I_{\cD_{A^*}}\otimes S_i)\widetilde
\Psi$, \ $i=1,\ldots, n, $ which proves that $\widetilde \Psi$ is a
multi-analytic operator. According to \cite{Po-holomorphic} (see
also \cite{Po-pluriharmonic}), $\widetilde \Psi$ is the boundary
function of a unique  bounded free holomorphic function $\Psi \in
\cS_{\bf ball}(B(\cN^\perp, \cD_{A^*}))$. The proof is complete.
\end{proof}

Using Lemma \ref{factorization}, we obtain the following refinement
of Theorem \ref{main}.
\begin{remark}\label{re1}
In Theorem \ref{main}, there is a one-to-one correspondence between
 the
noncommutative Schur class $\cS_{\bf ball}^{\Omega}(B(\cD_A,
\cD\oplus   \cD_A^{(n)}))$ and the Schur class $\cS_{\bf
ball}(B(\cG,\cD_{\Omega^*}))$, given by the formula
$$
\Psi=(\Omega P_\cF\otimes I)+(D_{\Omega^*}\otimes I)
\Psi_1(P_{\cG}\otimes I),
$$
 where $\Omega$ is defined by
\eqref{OmOm},  $\cG:=\cD_A\ominus \cF$,  and  $\Psi_1\in \cS_{\bf
ball}(B(\cG,\cD_{\Omega^*}))$. Consequently, Theorem \ref{main} can
be restated and the Schur class $\cS_{\bf ball}^{\Omega}(B(\cD_A,
\cD\oplus   \cD_A^{(n)}))$ can be replaced by $\cS_{\bf
ball}(B(\cG,\cD_{\Omega^*}))$.
\end{remark}

    The following result is an addition to  Theorem \ref{main}.

\begin{theorem}
\label{one-to-one1} Let $B$ be a solution of the {\bf GNCL} problem
with the data set $\{A,T,V,C,Q\}$, let $\Gamma$ be the contraction
determined by $B$, and  $\Theta \in H^2_{\bf ball} (B(\cD_A, \cD))$
 be  the free holomorphic function with symbol
$\Gamma$. Then the restriction of the map $J_\Gamma$ (defined by
\eqref{J-ga}) to $\cS_{\bf ball}^{\Lambda}(B(\cD_\Gamma,
\cD_\Gamma^{(n)}))$ is a one-to-one function onto   the set of all
functions
$\left[\begin{matrix} L\\
M_1\\
\vdots\\
M_n\end{matrix}\right]$ in  the Schur class  $\cS_{\bf
ball}^{\Omega}(B(\cD_A, \cD\oplus   \cD_A^{(n)}))$ and satisfying
the equation
\begin{equation*}
  \Theta(Z)=L(Z)\left[ I_{\cD_A\otimes \cZ}-
\sum_{i=1}^n(I_{\cD_A}\otimes Z_i)M_i(Z)\right]^{-1}\quad  \text{
for
  } \ Z:=(Z_1,\ldots, Z_n)\in [B(\cZ)^n]_1.
\end{equation*}
\end{theorem}
\begin{proof}
Since $B$ be a solution of the {\bf GNCL} problem, $\Gamma$
satisfies  relation \eqref{DSGA}. Then $\Gamma=\Gamma_\theta$,
where $\Theta$
is given as  above and $\Psi:=\left[\begin{matrix} L\\
M_1\\
\vdots\\
M_n\end{matrix}\right]$ is in  the noncommutative Schur class
$\cS_{\bf ball}^\Omega (B(\cD_A, \cD\oplus   \cD_A^{(n)}))$. Due to
Theorem \ref{one-to-one}, there exists a unique $\Phi\in \cS_{\bf
ball} (B(\cD_\Gamma, \cD_\Gamma^{(n)}))$ such that
$J_\Gamma\Phi=\Psi$. By Lemma \ref{main-ingredient}, we deduce that
$\Phi\in  \cS_{\bf ball}^\Lambda (B(\cD_\Gamma,  \cD_\Gamma^{(n)}))$
and the restriction of $J_\Theta$  to  $\cS_{\bf
ball}^{\Lambda}(B(\cD_\Gamma, \cD_\Gamma^{(n)}))$  is a one-to-one
function onto  the Schur class $\cS_{\bf ball}^{\Omega}(B(\cD_A,
\cD\oplus \cD_A^{(n)}))$. The proof is complete.
\end{proof}

Using Lemma \ref{factorization}, we obtain the following refinement
of Theorem \ref{one-to-one1}.
\begin{remark}\label{re2}
In Theorem \ref{one-to-one1}, there is a one-to-one correspondence
between
 the
noncommutative Schur class $\cS_{\bf ball}^{\Lambda}(B(\cD_\Gamma,
\cD_\Gamma^{(n)}))$ and the Schur class $\cS_{\bf
ball}(B(\cG_\Gamma,\cD_{\Lambda^*}))$, given by the formula
$$
\Phi=(\Lambda P_{\cF_\Gamma}\otimes I)+(D_{\Gamma^*}\otimes I)
\Phi_1 (P_{\cG_\Gamma}\otimes I),
$$
where $\Lambda$ is defined by \eqref{Lambd},
$\cG_\Lambda:=\cD_\Gamma \ominus \cF_\Gamma$,  and $\Phi_1\in
\cS_{\bf ball}(B(\cG_\Gamma,\cD_{\Lambda^*}))$. Consequently,
Theorem \ref{one-to-one1} can be restated and  the Schur class
$\cS_{\bf ball}^{\Lambda}(B(\cD_\Gamma,  \cD_\Gamma^{(n)}))$ can be
replaced by  $\cS_{\bf ball}(B(\cG_\Gamma,\cD_{\Lambda^*}))$.
\end{remark}

Now, we consider a few remarkable particular cases.
\begin{corollary}
\label{treil-Volberg} Let $\{A,T,V,C,Q\}$ be a data set. In the
particular case when  $\cG_i=\cX$ and $C_i=I_\cX$ for $i=1,\ldots,
n$, Theorem \ref{main} provides a description of all solutions of
the multivariable generalization \cite{Po-nehari} of Treil-Volberg
commutant lifting theorem \cite{TV}.
\end{corollary}

Let $T:=(T_1,\ldots, T_n)$, $T_i\in B(\cH)$, be a row contraction
and let $V:=(V_1,\ldots, V_n)$, $V_i\in B(\cK)$, be the minimal
isometric dilation of $T$ on a Hilbert space $\cK\supset \cH$.   Let
$Y:=(Y_1,\ldots, Y_n)$, $Y_i\in B(\cX)$, be a row isometry and let
$A\in B(\cX,\cH)$ be a contraction such that $ T_iA=AY_i$,\quad
$i=1,\ldots,n. $ The {\it noncommutative commutant lifting} ({\bf
NCL}) problem (see \cite{Po-isometric}) is to find $B\in B(\cX,
\cK)$ such that $\|B\|\leq 1$,
$$
P_\cH B=A,\ \text{ and } \ V_iB=BY_i,\quad i=1,\ldots,n.
$$
 A parametrization of all
solutions in the {\it noncommutative commutant lifting theorem}
({\bf NCLT}) in terms of generalized choice sequences was obtained
in \cite{Po-intert}.   Using   the   results of this section, we
obtain   a more refined parametrization and, moreover,  a concrete
Schur type description of all solutions.

\begin{theorem}
\label{NCLT} Let  $\{A,T,V,C,Q\}$ be a  data set in  the particular
case when, for each $i=1,\ldots,n$, $\cG_i=\cX$, $C_i=I_\cX$,
$Q_i=Y_i\in B(\cX)$,  and $Y:=(Y_1,\ldots, Y_n)$ is a row isometry.
Then any solution of the of the {\bf NCL} problem is given by
 \begin{equation}
 \label{B-rep2}
  B=\left[\begin{matrix}A\\
  \Gamma D_{A}
  \end{matrix}
  \right]: \cX\to \cH\oplus [\cD\otimes F^2(H_n)],
  \end{equation}
 where $\Gamma:\cD_{A}\to \cD\otimes  F^2(H_n)$ is the symbol of a   free holomorphic function
 $\Theta \in H^2_{\bf ball} (B(\cD_A, \cD))$
 given by
\begin{equation}
\label{Th-equ2}
  \Theta(Z)=L(X)\left[ I_{\cD_A\otimes \cZ}-
\sum_{i=1}^n(I_{\cD_A}\otimes Z_i)M_i(Z)\right]^{-1}\quad  \text{
for
  } \ Z:=(Z_1,\ldots, Z_n)\in [B(\cZ)^n]_1,
\end{equation}
where $\Psi:=\left[\begin{matrix} L\\
M_1\\
\vdots\\
M_n\end{matrix}\right]$  is an arbitrary element in  the
noncommutative Schur class $\cS_{\bf ball}^{\Omega}(B(\cD_A,
\cD\oplus   \cD_A^{(n)}))$. Moreover, the solution $B$ and the Schur
function $\Psi$ uniquely determine each other via the relations
\eqref{B-rep2} and \eqref{Th-equ2}. There is a unique solution to
the {\bf NCL} problem if and only if $\cF=\cD_A$ or $\Omega
\cF=\cD\oplus \cD_A^{(n)}$, where $\cF$ and $\Omega$ are defined by
\eqref{F} and \eqref{OmOm}, respectively.
\end{theorem}

\begin{proof}
The first part of the theorem follows from Theorem \ref{main}. To
prove the second part, let $B$ be a solution of the {\bf NCL}
problem, let $\Gamma$  be  the  contraction determined by $B$, and
$\Theta\in H^2_{\bf ball} (B(\cD_A, \cD))$ be the free holomorphic
function  with symbol  $\Gamma$.
   Due to Theorem \ref{one-to-one1}, to prove that $B$ and $\Psi$ uniquely determine each other,
    it is enough to show that
 $\cS_{\bf ball}^{\Lambda}(B(\cD_\Gamma,
\cD_\Gamma^{(n)}))$ has  just one element.  Since $(Y_1,\ldots,
Y_n)$ is a row isometry, the operator $\Omega$ (see  \eqref{OmOm})
is an isometry and, consequently, so is $\Lambda$. From the
definition of $\Lambda$ (see \eqref{Lambd}), we deduce that the
range of $\Omega$ coincides with $  \cD_\Lambda^{(n)}$. Therefore,
$\Lambda:\cF_\Gamma\to   \cD_\Lambda^{(n)}$ is a unitary operator.
Consequently, $\cD_{\Lambda^*}=\{0\}$. According to Remark
\ref{re2},  we deduce that $\cS_{\bf
ball}(B(\cG_\Gamma,\cD_{\Lambda^*}))$ is a singleton and, therefore,
so is $\cS_{\bf ball}^{\Lambda}(B(\cD_\Gamma, \cD_\Gamma^{(n)}))$.
Therefore, we have proved that any solution $B$ corresponds to a
unique Schur function $\Psi$.
Now, due to Remark \ref{re1}, there is a unique solution of the {\bf
NCL} problem if and only if $\cS_{\bf
ball}(B(\cG,\cD_{\Omega^*}))=\{0\}$. The latter equality holds if
and only if $\cG=\{0\}$ or $\cD_{\Omega^*}=\{0\}$. Since $\Omega$ is
an isometry, the condition $\cD_{\Omega^*}=\{0\}$ is equivalent to
$\Omega \cF=\cD\oplus \cD_A^{(n)}$. The proof is complete.
\end{proof}
We remark that one can easily obtain a version of Theorem \ref{NCLT}
in the more general  setting  of the {\bf NCL} problem when the row
isometry $(Y_1,\ldots, Y_n)$ is replaced by an arbitrary row
contraction.

 Another consequence of Theorem \ref{main} is the following
   multivariable version  \cite{Po-nehari} of
the weighted commutant lifting theorem of  Biswas-Foia\c s-Frazho
\cite{BFF}.

\begin{corollary}
\label{weighted} Let $T:=(T_1,\ldots, T_n)$, $T_i\in B(\cH)$, be a
row contraction and let $V:=(V_1,\ldots, V_n)$, $V_i\in B(\cK)$, be
the minimal isometric dilation of $T$ on a Hilbert space $\cK\supset
\cH$. Let  $A\in B(\cX,\cH)$ be such that $A^*A\leq P$, where $P\in
B(\cX)$ be a positive operator. Let $Y:=(Y_1,\ldots, Y_n)$, $Y_i\in
B(\cX)$, be such that \  $ [Y_i^*PY_j]_{n\times n}\geq [\delta_{ij}
P]_{n\times n} $  and
$$
T_iA=AY_i,\qquad i=1,\ldots,n.
$$
Then there exists  $B\in B(\cX, \cK)$ such that $P_\cH B=A$,
$B^*B\leq P$, and $$
 \ V_iB=BY_i,\quad i=1,\ldots,n.$$

\end{corollary}
\begin{proof}
Let $\widetilde X:=\overline{P^{1/2}\cX}$. Since $A^*A\leq P$, there
exists a unique contraction $\widetilde A:\widetilde X\to \cH$
satisfying the equation $A=\widetilde A P^{1/2}$. The condition
$T_iA=AY_i$, $ i=1,\ldots,n$, implies $ T_i\widetilde A
C_i=\widetilde A Q_i,\quad i=1,\ldots,n, $ where $C_i=P^{1/2}$ and
$Q_i=P^{1/2} Y_i:\cX\to \cX$ for $i=1,\ldots,n$. Applying Theorem
\ref{main}, we find a contraction $\widetilde B:\widetilde X\to \cK$
such that $P_\cH \widetilde B= \widetilde A$ and $V_i \widetilde B
C_i= \widetilde B Q_i$ for $i=1,\ldots, n$. Setting $B:=\widetilde B
P^{1/2}$, we have
$$
V_iB=V_i \widetilde B C_i= \widetilde B Q_i=\widetilde B P^{1/2} Y_i
=B Y_i
$$
for $i=1,\ldots,n$. Note also that $P_\cH B=P_\cH \widetilde B
P^{1/2}= \widetilde A P^{1/2} A$. Since $\widetilde B$ is a
contraction, we have $B^*B P^{1/2} \widetilde B^* \widetilde B
P^{1/2}\leq P$. This completes the proof.
\end{proof}

In a future paper, we apply the results of the present paper
  to the interpolation theory
setting to  obtain
   parametrizations and complete
 descriptions  of all solutions to the Nevanlinna-Pick,
 Carath\'eodory-Fej\' er, and Sarason type interpolation problems  for
  the noncommutative Hardy spaces $H^\infty_{\bf
 ball}$  and   $H^2_{\bf ball}$, as well as consequences to  (norm constrained)
 interpolation on the unit ball of $\CC^n$.

      %\Refs
      %\widestnumber\key{BFPQR}
      %\def\n{\key}
       %

      \end{document}